\documentclass[10pt,a4paper]{article}

\usepackage[T1]{fontenc}
\usepackage[all]{xy}
\usepackage{amsthm}
\usepackage{amssymb}
\usepackage{amsmath}
\usepackage{bbm}
\usepackage{graphics,graphicx}
\usepackage{ esint }
\usepackage{eepic}
\usepackage{epsfig}
\usepackage{times}
\usepackage{stmaryrd}
\usepackage{hyperref}
\usepackage{array}

\usepackage[usenames]{color}
\usepackage[toc,page]{appendix}

\DeclareMathOperator{\cotan}{cotan}
\usepackage{tikz}
\usetikzlibrary{calc,patterns,angles,quotes}
\usetikzlibrary{patterns}

\newtheorem{theorem}{Theorem}[]
\newtheorem{lemma}[theorem]{Lemma}
\newtheorem{Proposition}[theorem]{Proposition}
\newtheorem{definition}[theorem]{Definition}
\newtheorem{remark}[theorem]{Remark}

\newcolumntype{M}[1]{>{\raggedright}m{#1}}

\usepackage{caption} 
\usepackage{tikz}
\newcommand{\R}{\mathbb{R}}
\newcommand{\del}{\delta}
\newcommand{\N}{\mathbb{N}}

\renewcommand{\S}{\mathcal{S}}

\newcommand{\K}{\mathcal{K}}
\newcommand{\p}{\mathcal{P}}
\newcommand{\D}{\mathcal{D}}
\newcommand{\Om}{\Omega}
\newcommand{\eps}{\varepsilon}

\providecommand{\keywords}[1]{\textbf{\textit{Keywords:}} #1}
\providecommand{\ams}[1]{\textbf{\textit{AMS classification:}} #1}

\newcounter{mnotecount}[section]

\newcommand{\rmnote}[1]{}

\title{Complete systems of inequalities relating the perimeter, the area and the Cheeger constant of planar domains}
 \author{Ilias Ftouhi}
 \date{
    \today
}

\begin{document}
\maketitle






\begin{abstract}

The object of the paper is to find complete systems of inequalities relating the perimeter $P$, the area $|\cdot|$ and the Cheeger constant $h$ of planar sets. To do so, we study the so called Blaschke--Santal\'o diagram of the triplet $(P,h,|\cdot|)$ for different classes of domains: simply connected sets, convex sets and convex polygons with at most $N$ sides. We completely determine the diagram in the latter cases except for the class of convex $N$-gons when $N\ge 5$ is odd: therein, we show that the boundary of the diagram is given by the graphs of two continuous and strictly increasing functions. An explicit formula for the lower one and a numerical method to obtain the upper one is provided. At last, some applications of the results are presented. 
\end{abstract}


\ams{52A10, 52A40, 65K15.}











\section{Introduction and main results}\label{s:introduction}

Let $\Om$ be a bounded subset of $\R^n$ (where $n\ge2$). The Cheeger problem consists in studying the following minimization problem
\begin{equation}\label{def:cheeger}
h(\Om) := \inf\left\{\ \frac{P(E)}{|E|}\ \Big|\ E\ \text{measurable and}\ E\subset \Om\right\} ,    
\end{equation}
where $P(E)$ is the distributional perimeter of $E$ measured with respect to $\R^n$ (see for example \cite{MR2832554} for definitions) and $|E|$ is the $n$-dimensional Lebesgue measure of $E$. The quantity $h(\Om)$ is called \textit{the Cheeger constant of $\Om$} and any set $C_\Om \subset \Om$ for which the infimum is attained is called a \textit{Cheeger set of $\Om$.} \vspace{2mm}

Since the early work of Jeff Cheeger \cite{jeff_cheeger}, many authors took interest in the Cheeger problem has interested various authors. An introductory survey on the subject is referenced in \cite{MR2832554}. We recall that every bounded domain $\Om$ with Lipschitz boundary admits at least one Cheeger set $C_\Om$, see for example \cite[Proposition 3.1]{MR2832554}. In \cite{cham_1}, the authors prove the uniqueness of the Cheeger set when $\Om\subset \R^n$ is convex, but as far as we know there is no complete characterization of $C_\Om$ in dimension $n\ge3$ (even when convexity is assumed), in contrast to the planar case which was treated by Bernd Kawohl and Thomas Lachand-Robert in \cite{kawohl} where a complete description of the Cheeger sets of planar convex domains is given in addition to an algorithm to compute the Cheeger constant of convex polygons. We finally refer to \cite{MR3719067,MR3451406,MR4156828} for recent results in larger classes of sets. \vspace{2mm}


In this paper, we aim to describe all possible geometrical inequalities involving the perimeter, the area and the Cheeger constant of a given planar shape. It comes down to studying a so called Blaschke--Santal\'o diagram of the triplet  $(P,h,|\cdot|)$. \vspace{2mm}

 Such a diagram allows to visualize all possible inequalities between three geometrical quantities. It is named after Blaschke and Santal\'o in reference to their works \cite{blaschke,santalo}, where they first introduced this notion.   Afterward, these diagrams have been extensively studied, especially for the class of planar convex sets. We refer to \cite{cifre} for more details and various examples involving geometrical functionals.  We also refer to \cite{buttazzo_3,delyon,delyon2,ftouh,LZ,buttazzo_2,vdBBP} for some recent results. 

For more precision, let us define the Blaschke--Santal\'o diagrams we are interested in in this paper.
\begin{definition}\label{def:diagrams}
Given $\mathcal{F}$ a class of measurable planar sets of, we define 
\begin{eqnarray*}\mathcal{D}_{\mathcal{F}}
&:=&\Big\{(x,y)\in\R^2, \;\;\exists\ \Om\in\mathcal{F}\textrm{ such that }|\Om|=1, P(\Om)=x,\ h(\Om)=y\Big\}\\
&:=&\Big\{\ \big(P(\Omega),h(\Omega)\big),\ \Omega\in\mathcal{F}, \;|\Omega|=1\ \Big\}.
\end{eqnarray*}
\end{definition}

We note that thanks to the following { scaling properties:} 
$$\forall t>0, \;\;\;\; h(t\Om) = \frac{h(\Om)}{t},\ \ \ \ |t\Om|=t^2 |\Om|\ \ \ \ \ \text{and}\ \ \ \ \ P(t\Om)=tP(\Om),$$
one can give a { scale-invariant} formulation of the diagram
\begin{eqnarray*}\mathcal{D}_{\mathcal{F}}
&=&\left\{(x,y)\in\R^2, \;\;\exists \Om\in\mathcal{F}\textrm{ such that } P(\Om)/|\Om|^{1/2}=x,\ \  |\Om|^{1/2}h(\Om)=y\right\}\\
&=&\left\{\ \left(\frac{P(\Om)}{|\Om|^{1/2}},|\Om|^{1/2}h(\Omega)\right),\ \ \Omega\in\mathcal{F} \right\}.
\end{eqnarray*}

In the whole paper, we denote { by}:
\begin{itemize}
    \item $\S^2$ the set of bounded, planar and non-empty simply connected sets,
    \item $\K^2$ the set of bounded, planar and non-empty convex sets,
    \item $\p_N$ the set of convex polygons of at most $N$ sides,
    \item $B$ the disk of unit area,
    \item $R_N$ a regular polygon of $N$ sides and unit area,
    \item $d^H$ the Hausdorff distance, see for example \cite[Chapter 2]{HP05Var} for { the} definition and more details,
    \item $d(\Om)$ and $r(\Om)$ respectively the diameter and inradius of the set $\Om$.
\end{itemize}


We are aiming at describing all possible inequalities relating $P$, $|\cdot|$ and $h$ for different classes of planar sets and thus describing the associated Blaschke--Santal\'o diagrams. Let us first state the inequalities that we already know; if $\Om$ is measurable, we have :

\begin{itemize}
\item the isoperimetric inequality: 
\begin{equation}\label{eq:isoperimetric}
\frac{P(\Om)}{|\Om|^{1/2}}\ge \frac{P(B)}{|B|^{1/2}}=2\sqrt{\pi},
\end{equation}
\item a consequence of the definition of the Cheeger constant
\begin{equation}\label{eq:trivial}
h(\Om) = \inf_{E\subset \Om} \frac{P(E)}{|E|}\leq \frac{P(\Om)}{|\Om|},
\end{equation}
\item a Faber-Krahn type inequality: 
\begin{equation}\label{eq:faber-krahn}
|\Om|^{1/2}h(\Om)\ge |B|^{1/2}h(B)=\frac{P(B)}{|B|^{1/2}}=2\sqrt{\pi}.
\end{equation}
\end{itemize}
{ We note that each inequality may be visualized in the Blaschke--Santal\'o diagram as the  hypograph or subgraph of a given function, see Figure \ref{fig:diagrams} for example.}
\vspace{2mm}

It is natural to wonder if there are other inequalities. Yet, the answer is tightly related to the choice of the class of sets $\mathcal{F}$. In the present paper, we are interested in studying complete systems of inequalities relating the perimeter $P$, the Cheeger constant $h$ and the area $|\cdot|$ for three classes of planar sets: \begin{enumerate}
    \item The class of simply connected sets.
    \item The class of convex sets.
    \item The classes of convex polygons of at most $N$ sides, where $N\ge 3$. 
\end{enumerate}


\subsection{Results on the classes of simply connected and convex set}

We provide the  complete descriptions of the Blaschke--Santal\'o diagrams of the triplet $(P,h,|\cdot|)$ for both the classes $\mathcal{S}^2$ of planar simply connected sets and $\K^2$ of planar convex sets.

\begin{theorem}\label{th:diagram_open}
We take $x_0=P(B) = 2\sqrt{\pi}$. 
\begin{enumerate}
\item The diagram of the class $\mathcal{S}^2$ of planar simply connected domains is given by
$$\D_{\mathcal{S}^2}=\{(x_0,x_0)\}\cup\left\{(x,y)\ \ |\ \ x> x_0\ \ \ \text{and}\ \ \ x_0<y\leq x\right\}.$$
\item 
The diagram of the class $\mathcal{K}^2$ of planar convex domains is given by
$$\D_{\mathcal{K}^2}=\left\{(x,y)\ \ \Big|\ \ x\ge x_0\ \ \ \text{and}\ \ \  \frac{x}{2}+\sqrt{\pi}\leq y \leq x \right\}.$$
\end{enumerate}
\end{theorem}

{ 
The study of the boundary of the Blaschke--Santal\'o diagram $\D_{\K^2}$ is a crucial step toward its complete description. Thus, we introduce the following definition of upper and lower boundaries:
\begin{definition}\label{def:boundary}
We take $x_0=P(B) = 2\sqrt{\pi}$. We define the upper boundary of the diagram $\D_{\K^2}$ as follows: 
$$\{(x,y_x)\ |\ x\ge x_0\},$$
where $$y_x:= \sup\{h(\Omega)\ |\ \Om\in \K^2,\ P(\Om)=x\ \text{and}\ |\Om|=1\}\in \R,$$
and the lower boundary of the diagram $\D_{\K^2}$ as follows: 
$$\{(x,y_x)\ |\ x\ge x_0\},$$
where $$y_x:= \inf\{h(\Omega)\ |\ \Om\in \K^2,\ P(\Om)=x\ \text{and}\ |\Om|=1\}\in \R.$$
\end{definition}}

\begin{center}
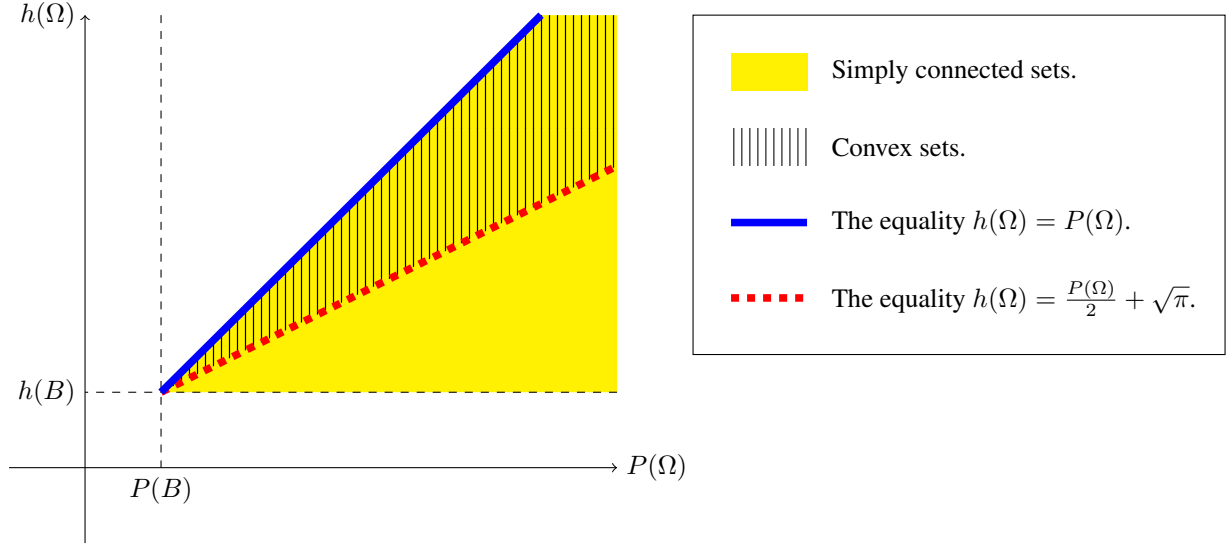

\begin{tikzpicture}
\fill[color=yellow] (8.5,5) -- (9.5,5) -- (9.5,5.5) -- (8.5,5.5) -- cycle; 
\draw (9.7,5.25) node[right] {Simply connected sets.};
\fill[color=gray!20,pattern=vertical lines] (8.5,4) -- (9.5,4) -- (9.5,4.5) -- (8.5,4.5) -- cycle; 
\draw (9.7,4.25) node[right] {Convex sets.};
\draw [line width = 1mm, blue] (8.5,3.25) -- (9.5,3.25);
\draw (9.7,3.25) node[right] {The equality $h(\Om)=P(\Om)$.};
\draw [line width = 1mm, dashed, red] (8.5,2.25) -- (9.5,2.25);
\draw (9.7,2.25) node[right] {The equality $h(\Om)=\frac{P(\Om)}{2}+\sqrt{\pi}$.};

\fill[color=yellow] (1,1) -- (7,1) -- (7,6) -- (6,6) -- cycle; 
\fill[color=gray!20,pattern=vertical lines] (1,1) -- (7,4) -- (7,6) -- (6,6) -- cycle;
\draw[->] (-1,0) -- (7,0);
\draw (7,0) node[right] {$P(\Omega)$};
\draw [->] (0,-1) -- (0,6);
\draw (0,6) node[left] {$h(\Omega)$};
\draw [dashed] (1,6) -- (1,0) node[below] {$P(B)$};
\draw [dashed] (7,1) -- (0,1) node[left] {$h(B)$};
\draw [domain=1:7, dashed, line width = 1mm, red] plot(\x,{.5*\x+.5});
\draw [domain=1:6 , line width = 1mm, blue] plot(\x,{\x});
\draw (8,1.5) -- (15,1.5) -- (15,6) -- (8,6) -- cycle;

\end{tikzpicture}
\captionof{figure}{The Blaschke--Santal\'o diagrams for the classes of simply connected sets and convex sets.}
\label{fig:diagrams}
\end{center}


We note that by taking advantage of the inequalities \eqref{eq:isoperimetric} and \eqref{eq:faber-krahn}, it is also classical to represent the Blaschke--Santal\'o diagrams as subsets of $[0,1]^2$. In our situation, this means to consider the sets
$$\D'_\mathcal{F}:= \left\{\left(X,Y\right)\ |\ \exists \Om\in \mathcal{F}\ \text{such as}\ |\Om|=1\ \text{and}\ \left(X,Y\right)=\left(\frac{P(B)}{P(\Om)},\frac{h(B)}{h(\Om)}\right)\right\}\subset [0,1]^2,$$
where $\mathcal{F}$ is a given class of planar sets. With this parametrization, the Blaschke--Santal\'o diagrams for the classes $\S^2$ and $\K^2$ are given by the following sets:
 $$
\left \{
\begin{array}{c @{=} c}
    \D'_{\mathcal{S}^2}\ \ \ &\ \ \ \{(1,1)\}\cup\{(X,Y)\ |\ X\in (0,1)\ \text{and}\ X\leq Y<1\}, \vspace{3mm}\\ 
    \D'_{\mathcal{K}^2}\ \ \ &\ \{(X,Y)\ |\ X\in (0,1]\ \text{and}\ X\leq Y\leq \frac{2X}{1+X}\},
\end{array}
\right.
    $$
    which are represented in Figure \ref{fig:diagrams_0_1}. 

\begin{center}
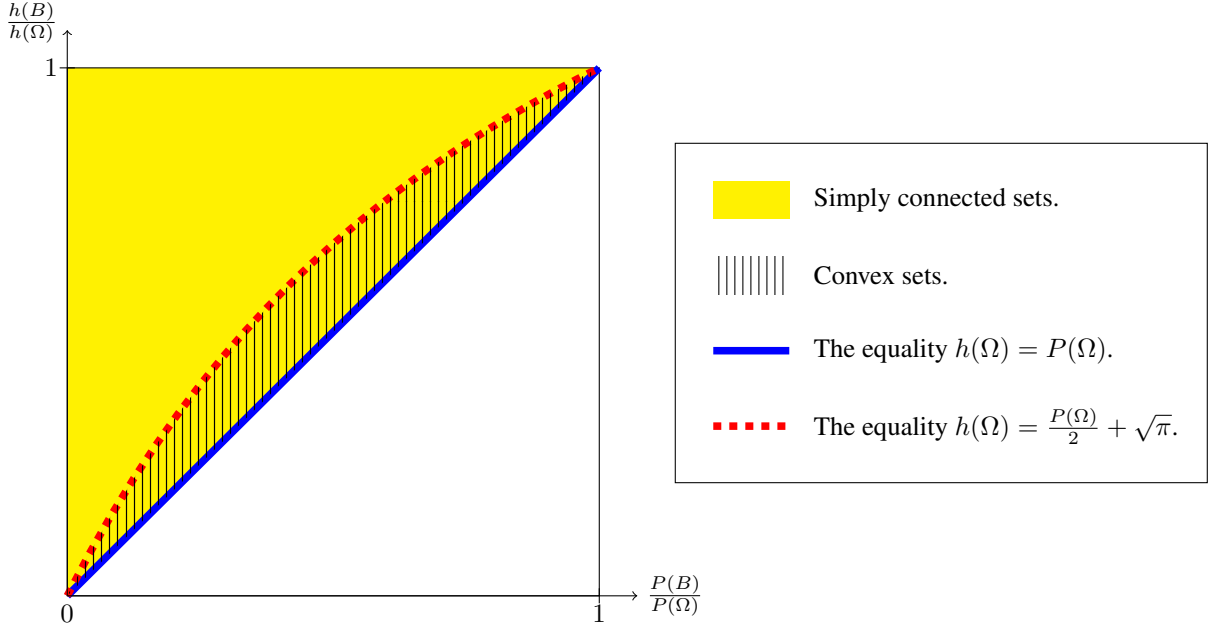

\begin{tikzpicture}
\fill[color=yellow] (8.5,5) -- (9.5,5) -- (9.5,5.5) -- (8.5,5.5) -- cycle; 
\draw (9.7,5.25) node[right] {Simply connected sets.};
\fill[color=gray!20,pattern=vertical lines] (8.5,4) -- (9.5,4) -- (9.5,4.5) -- (8.5,4.5) -- cycle; 
\draw (9.7,4.25) node[right] {Convex sets.};
\draw [line width = 1mm, blue] (8.5,3.25) -- (9.5,3.25);
\draw (9.7,3.25) node[right] {The equality $h(\Om)=P(\Om)$.};
\draw [line width = 1mm, dashed, red] (8.5,2.25) -- (9.5,2.25);
\draw (9.7,2.25) node[right] {The equality $h(\Om)=\frac{P(\Om)}{2}+\sqrt{\pi}$.};
\draw (8,1.5) -- (15,1.5) -- (15,6) -- (8,6) -- cycle;
\fill[color=yellow] (0,0) -- (7,7) -- (0,7) -- cycle; 
\draw (0,0) -- (7,0) -- (7,7) -- (0,7) -- cycle;
\draw [domain=0:7 , line width = 1mm, blue] plot(\x,{\x});
\draw [domain=0:7 , line width = 1mm, dashed, red] plot(\x,{14*(\x/7)/(1+\x/7)});
\fill [pattern=vertical lines] (0,0) -- plot [domain=0:7] (\x,{14*(\x/7)/(1+\x/7)}) -- (7,7) -- cycle;
\draw [->] (0,0)--(7.5,0);
\draw [->] (0,0)--(0,7.5);
\draw (7.5,0) node[right] {$\frac{P(B)}{P(\Omega)}$};
\draw (7,0) node[below] {$1$};
\draw (0,7) node[left] {$1$};
\draw (0,0) node[below] {$0$};
\draw (0,7.6) node[left] {$\frac{h(B)}{h(\Omega)}$};
\draw  (7,-.1)--(7,0.1);
\draw  (-.1,7)--(0.1,7);
\end{tikzpicture}
\captionof{figure}{The Blaschke--Santal\'o diagrams for the classes of simply connected sets and convex sets represented in $[0,1]^2$.}
\label{fig:diagrams_0_1}
\end{center}

Let us give some comments on the latter results:
\begin{itemize}
    \item One major step in the study of the diagram of convex sets is to prove the following sharp inequality

    \begin{equation}\label{eq:new_inequality}
    \forall \Om \in \K^2,\ \ \ \ \  h(\Om)\ge \frac{P(\Om)+\sqrt{4\pi |\Om|}}{2|\Om|},
    \end{equation}
    where equality occurs for example for circumscribed polygons (i.e., those whose sides touch their incircles) and more generally for sets which are homothetical to their form bodies\footnote{We refer to  \cite[Page 386]{schneider} for the definition of form bodies.}. 
    \item The inequality \eqref{eq:new_inequality} is rather easy to prove when the convex $\Om$ is a Cheeger-regular polygon (that is, its Cheeger set touches all of its sides), see \cite[Remark 32]{bucur_fragala}, but much difficult to prove for general convex sets as shown in the present paper (see Section \ref{ss:proof_of_new_inequality}). We also note that this inequality may be seen as a quantitative isoperimetric inequality for the Cheeger constant of convex planar sets. Indeed, it can be written in the following form
    $$\forall \Om \in \K^2,\ \ \ \ \  |\Om|^{1/2}h(\Om) - |B|^{1/2}h(B)\ge \frac{1}{2}\left(\frac{P(\Om)}{|\Om|^{1/2}}-\frac{P(B)}{|B|^{1/2}}\right)\ge 0.$$
    We refer to \cite{MR2480287,julin2020quantitative} for some examples of quantitative inequalities for the Cheeger constant. 
    
    Moreover, we note in Section \ref{ss:improvement_of_inequalities} that the inequality \eqref{eq:new_inequality} is stronger than a classical result \cite[Theorem 3]{MR883424} due to R. Brooks and P. Waksman. It also improves in the planar case a more recent estimate given in \cite[Corollary 5.2]{brasco_makai}, which states that for any open, bounded and convex set $\Om\subset\R^n$, where $n\ge2$, one has
    $$h(\Om)\ge \frac{1}{n}\cdot \frac{P(\Om)}{|\Om|}.$$
    \item We note that the first statement of Theorem \ref{th:diagram_open} asserts that the inequalities \eqref{eq:trivial} and \eqref{eq:faber-krahn} form a complete system of inequalities of the triplet $(P,h,|\cdot|)$ in any class of planar sets that contains $\mathcal{S}^2$. Meanwhile, the second one asserts that this is no longer the case for the class $\K^2$ of planar and convex sets, where estimates \eqref{eq:trivial} and \eqref{eq:new_inequality} are shown to be forming a complete system of inequalities for the triplet $(P,h,|\cdot|)$.
    \item { One could wonder why we chose to work with the class of simply connected sets. The main reason is that for any subclass of measurable domains that contains the simply connected ones, the diagram is the same. Indeed, if we denote by $\mathcal{C}^2$ a subclass of planar and measurable sets that contains the class $\S^2$, we have by the inequalities \eqref{eq:trivial} and \eqref{eq:faber-krahn} (where the equality holds in \eqref{eq:faber-krahn} only for balls) 
    $$\D_{\mathcal{C}^2}\subset \{(x_0,x_0)\}\cup\left\{(x,y)\ \ |\ \ x> x_0\ \ \ \text{and}\ \ \ x_0<y\leq x\right\} = \D_{\S^2}.$$
    Moreover, the inclusion $\S^2\subset\mathcal{C}^2$ implies that $\D_{\S^2}\subset \D_{\mathcal{C}^2}$, which proves the equality. }
    \item We finally note that due to technical convenience, we first show the second assertion (the case of convex sets) and then use it to prove the first one (the case of simply connected sets). 
\end{itemize}

\subsection{Results on the classes of convex polygons}

Now, let us focus on the class of convex polygons. We give an improvement of the inequality \eqref{eq:trivial} in the class $\p_N$ of convex polygons with at most $N$ sides,  where $N\ge 3$. We recall that since triangles are { circumscribed} polygons, one has
\begin{equation}\label{eq:triangles_cheeger}
    \forall \Om\in \p_3,\ \ \ h(\Om) = \frac{P(\Om)+\sqrt{4\pi|\Om|}}{2|\Om|},
\end{equation}
see the discussion below \cite[Theorem 3]{kawohl}. 

As for the case $N\ge 4$, we prove the following sharp inequality
\begin{equation}\label{eq:uper_bound}
\forall \Om\in \p_N,\ \ \ \ \ \ \ h(\Om)\leq \frac{P(\Om)+\sqrt{P(\Om)^2+4\big(\pi-N\tan{\frac{\pi}{N}}\big) |\Om|}}{2|\Om|},
\end{equation}
with equality if and only if $\Om$ is Cheeger-regular { (i.e., all its sides touch its Cheeger set $C_\Om$)} and all of its angles are equal (to $(N-2)\pi/N$). The equality is also asymptotically attained  by the following family $\big([0,1]\times [0,d]\big)_{d\ge1}$ of rectangles when $d$ goes to infinity. \vspace{2mm}

{ In order to simplify the notation, we denote by $\D_N:=\D_{\p_N}$ the Blaschke--Santal\'o diagram of the triplet $(P,h,|\cdot|)$ for the class of convex polygons of at most $N$ sides, see Definition \ref{def:diagrams} for the notion of Blaschke--Santal\'o diagrams. \vspace{2mm}

As for the case of convex sets, we introduce the notion of upper and lower boundaries of the diagram $\D_N$. 
\begin{definition}\label{def:boundary_polygons}
Let $N\ge 3$, we recall that $R_N$ corresponds to a regular polygon of $N$ sides and unit area. We define the upper boundary of the diagram $\D_N$ as follows: 
$$\{(x,g_N(x))\ |\ x\ge P(R_N)\},$$
where $$g_N:x\longmapsto  \sup\{h(\Omega)\ |\ \Om\in \p_N,\ P(\Om)=x\ \text{and}\ |\Om|=1\}\in \R,$$
and the lower boundary of the diagram $\D_N$ as follows: 
$$\{(x,y_x)\ |\ x\ge P(R_N)\},$$
where $$y_x:= \inf\{h(\Omega)\ |\ \Om\in \p_N,\ P(\Om)=x\ \text{and}\ |\Om|=1\}\in \R.$$

\end{definition}

\begin{theorem}\label{th:diagram_polygon}
Let $N\ge 3$. The diagram $\D_N:=\D_{\p_N}$ contains its upper and lower boundaries. Moreover, the lower boundary is given by the half line 
$$\left\{\left(x,\frac{x}{2}+\sqrt{\pi}\right)\ |\ x\ge P(R_N)\right\},$$
and the upper boundary is given by the graph of the continuous and strictly increasing function $g_N$
$$\left\{\left(x,g_N(x)\right)\ |\ x\ge P(R_N)\right\}.$$
Moreover, we have the following cases:
\begin{itemize}
\item if $N=3$, we have 
$$\D_3=\left\{\left(x,\frac{x}{2}+\sqrt{\pi}\right)\Big|\ \ x\ge P(R_3)\right\}.$$
\item If $N$ is even, then 
$$\mathcal{D}_N = \left\{(x,y)\ |\ x\ge P(R_N)\ \ \text{and}\ \ \frac{x}{2}+\sqrt{\pi}\leq y\leq g_N(x)\right\},$$
and 
$$\forall x\ge P(R_N),\ \ \ \ g_N(x) = \frac{x+\sqrt{x^2+4(\pi-N\tan{\frac{\pi}{N}})}}{2}.$$
\item If $N\ge 5$ is odd, there exist $c_N\ge b_N>P(R_N)$ such that $$\forall x\in [P(R_N),b_N],\ \ \  g_N(x)=\frac{x+\sqrt{x^2+4(\pi-N\tan{\frac{\pi}{N}})}}{2}$$ and $$\forall x\in\big[c_N,+\infty\big),\ \ \ g_N(x) < \frac{x+\sqrt{x^2+4(\pi-N\tan{\frac{\pi}{N}})}}{2}.$$
Moreover $$g_N(x)\underset{x\rightarrow +\infty}{\sim}x.$$
\end{itemize}

\end{theorem}

Let us give some comments on the latter results: 
\begin{itemize}
    \item For every $N\ge 3$, we have the inclusion $\D_N\subset\D_{N+1}$ (because $\p_N\subset \p_{N+1})$. Moreover, we notice that the right hand side of the inequality \eqref{eq:uper_bound}  (monotonically) converges to the right hand side of the inequality \eqref{eq:trivial}. Thus, one can recover the diagram $\D_{\K^2}$ of convex sets as the limit of $\D_N$. In fact, one can show that 
    $$\D_{\K^2} = \overline{\bigcup_{N=3}^{+\infty} \D_N},$$
    where $\overline{\bigcup_{N=3}^{+\infty} \D_N}$ corresponds to the closure of $\bigcup_{N=3}^{+\infty} \D_N$.
    \item It is interesting to note that Theorem \ref{th:diagram_polygon} shows that the inequalities  \eqref{eq:new_inequality} and \eqref{eq:uper_bound} form a complete system of inequalities of the triplet $(P,h,|\cdot|)$ in the class $\p_N$ if and only if $N$ is even. 
    \item We believe that $c_N = b_N$ for every $N$ odd and greater or equal than $5$. To prove it, one could be tempted to solve the shape optimization problems 
    $$\max\{h(\Om)\ |\ |\Om|=1,\ P(\Om)=p_0\ \text{and}\ \Om\in \p_N\},$$
    where $p_0\in [P(R_N),+\infty)$. Unfortunately, this seems to be quite challenging in this case, since numerical simulations (see Section \ref{s:numeric}) suggest that the solutions are not unique and for large values of $p_0$ they are not necessarily Cheeger-regular (i.e., all its sides touch its Cheeger set $C_\Om$) and thus one does not have an explicit formula for their Cheeger constants, see Theorem \ref{th:KLR_polygon} and Lemma \ref{lem:jimmy}. Another possible strategy to prove the equality $c_N=b_N$ could be to show that one can continuously deform any Cheeger-regular $N$-gon $\Om$ whose angles are all equal into a regular $N$-gon while making sure to preserve these two properties throughout the process.
    \item We note that we are able to prove the following estimate $$c_N\leq 2N\sqrt{\tan{\frac{(N-2)\pi}{2N}}},$$
    see \eqref{ineq:c_N}.

\end{itemize}
}

The present paper is organized as follows: Section \ref{s:lemmas} contains two subsections, in the first one we recall some classical definitions and results needed for the proofs, in the second one we state and prove some preliminary lemmas which are also interesting { on their own}. The proofs of the main theorems are given in Section \ref{s:proofs}. Then, we provide some numerical simulations on the diagrams $\D_N$ in Section \ref{s:numeric}. Finally, we give some applications of the results of our results in Section \ref{s:remarks}. 

\section{Classical results and preliminaries}\label{s:lemmas}
\subsection{Classical { definitions and} results}
In this subsection, we recall some classical { definitions and} results that are used in the paper. 
\begin{theorem}\cite[Theorem 2 and Remark 3]{web}\label{th:T_estimates}\\
Take $N\ge 3$ and $\Om\subset \R^2$ a convex polygon of $N$ sides. We define: 
\begin{equation}\label{eq:def_T}
T(\Om) := \sum_{i=1}^N \frac{1}{\tan \frac{\alpha_i}{2}},    
\end{equation}
where $\alpha_i\in (0,\pi]$ are the interior angles of $\Om$.  We have the following estimates: 
\begin{equation}\label{eq:fragala}
N \tan{\frac{\pi}{N}} \leq T(\Om) \leq \frac{P(\Om)^2}{4|\Om|}.
\end{equation}
The lower bound is attained if and only if all the angles $\alpha_i$ are equal (to $\frac{N-2}{2N}\pi$), meanwhile the upper one is an equality if and only the polygon $\Om$ is circumscribed. 
\end{theorem}
\begin{remark}
The lower bound is a simple application of Jensen's inequality to the function $\cotan$ which is strictly convex on $(0,\pi/2)$. On the other hand, since $N \tan \frac{\pi}{N}>\pi$, the upper estimate may be seen as an improvement of the isoperimetric inequality for convex polygons. We refer for example to \cite{web} for a detailed proof of Theorem \ref{th:T_estimates}. 
\end{remark}

{ 
\begin{definition}\label{def:minkowski_sums}
Let $A$ and $B$ be two subsets of $\mathbb{R}^2$ and $t> 0$, we define
\begin{itemize}
    \item the Minkowski sum of the sets $A$ and $B$ by $$A\oplus B:= \{x+y\ |\ x\in A\ \text{and}\ y\in B\},$$
    \item and the dilatation of the set $A$ by the positive coefficient $t$ by $$t A:= \{t x\ |\ x\in A\}.$$ 
\end{itemize}
\end{definition}}

{
\begin{definition}\label{def:inner}
For a given planar set $\Omega$ and any $x\in \Omega$, we denote by $\textrm{dist}(x,\partial \Omega)$ the distance from $x$ to the boundary of $\Omega$, and for $t\ge 0$, we denote by $\Omega_{-t}$ the sets of points of $\Om$ of distance at least $t$ to the boundary, i.e., $$\Omega_{-t}:= \{x\in \Omega|\ \textrm{dist}(x,\partial \Omega)\ge t\}.$$ 
The sets $(\Omega_{-t})_{t\ge 0}$ are known as the inner parallel sets of $\Omega$.
\end{definition}}
Let us now recall some classical and important results on the Cheeger problem for planar convex sets. 
\begin{theorem}\label{th:KLR_convex}\cite[Theorem 1]{kawohl}
{ Let $\Omega$ be a planar convex set.} There exists a unique value $t=t^*>0$ such that $|\Om_{-t}|=\pi t^2$. We also have $h(\Om) = 1/t^*$. Moreover, the Cheeger set of $\Om$ is given by { $$C_\Om = \Om_{-t^*}\oplus t^*B_1,$$} where $B_1$ is the unit disk. 
\end{theorem}

\begin{theorem}\cite[Theorem 3]{kawohl}\label{th:KLR_polygon}
If $\Om$ is a Cheeger-regular polygon (i.e., all its sides touch its Cheeger set $C_\Om$), then,
\begin{equation}\label{eq:cheeger_equality}
h(\Om) = \frac{P(\Om)+\sqrt{P(\Om)^2-4\big(T(\Om)-\pi\big)|\Om|}}{2|\Om|}, 
\end{equation}
{ where $T(\Omega)$ is defined in \eqref{eq:def_T}.}
\end{theorem}
It is natural to wonder if { the equality \eqref{eq:cheeger_equality} holds for general convex polygons}. In Lemma \ref{lem:jimmy}, we prove that there is only an inequality and that the equality occurs if and only if the polygon is Cheeger-regular. \vspace{2mm}

{ Let us now recall some classical parametrizations of starshaped sets (or convex ones in particular). We refer to \cite[Section 1.7]{schneider} for more details. 
\begin{definition}\label{def:starshaped}
A set $\Omega \subset \R^2$ is called starshaped with respect to a point $A\in \R^2$, if it is not empty and for every $M\in \Omega$, the segment $[AM]$ is included in $\Omega$. 
\end{definition}

Let us now recall some important functions that allow to parametrize a given set $\Omega$. If $\Omega \subset \R^2$ is a compact and starshaped with respect to the origin $0$, we recall that 
\begin{itemize}
    \item its radial function is given by $f_\Om:x\in \R^2\backslash\{0\}\longmapsto f_\Om(x) = \max\{\lambda\ge 0\ |\ \lambda x \in \Om \}$, 
    \item and its gauge function is given by $u_\Om:\R^2\backslash\{0\}\longmapsto\inf\{\lambda\ge 0\ |\ x\in \lambda \Omega\}$.  
\end{itemize}
If $\Omega$ is convex (not necessarily containing the origin), its support function is defined as follows:
$$h_\Om:x\in \R^2\longmapsto \sup\{\langle x,y \rangle\ |\ y\in \Omega\}.$$
Since the functions $f_\Om$, $u_\Om$ and $h_\Om$ satisfy the following scaling properties $f_\Om(t x) = t f_\Om(x)$, $u_\Om(t x) = t^{-1} u_\Om(x)$ and $h_\Om(t x) = t h_\Om(x)$ for $t>0$, they can be characterized by their values on the unit sphere $\mathbb{S}^{1}$ or equivalently on the interval $[0,2\pi)$. 

In the present paper, the radial, gauge and support functions are defined on the interval $[0,2\pi)$ as stated in the following definition.  
\begin{definition}\label{def:parametrizations}
If $\Omega \subset \R^2$ is a compact and starshaped with respect to the origin $0$, 
\begin{itemize}
    \item the radial function of $\Om$ is given by $f_\Om:\theta\in [0,2\pi)\longmapsto \max\{\lambda\ge 0\ |\ \lambda \binom{\cos{\theta}}{ \sin{\theta}}\in \Omega\}$, 
    \item the gauge function of $\Om$ is given by  $u_\Om:[0,2\pi)\longmapsto \inf\{\lambda\ge 0\ |\ \binom{\cos{\theta}}{ \sin{\theta}}\in \lambda \Omega \}$.
\end{itemize}
If $\Omega$ is convex (not necessarily containing the origin), its support function is given by $$h_\Om:[0,2\pi)\longmapsto \sup\left\{\left\langle \binom{\cos{\theta}}{\sin{\theta}}, y \right\rangle\ |\ y\in\Om\right\}.$$
\end{definition}
Let us now give some important remarks on these functions. 
\begin{remark}\label{rk:parametrizations}
\begin{itemize}
\item  The functions introduced above can be defined in higher dimensions. We refer for example to \cite[Section 1.7]{schneider} for more details and results. 
    \item The gauge function is the inverse of the radial function. Indeed, for every $\theta\in [0,2\pi)$, we have $u_\Om(\theta) = 1/f_\Om(\theta)$.
    \item It is interesting to note that the gauge and support functions allow to provide a simple characterization of the convexity of a set $\Om$. Indeed, $\Om$ is convex if and only if $h_\Om''+h_\Om\ge 0$, which is also equivalent to $u_\Om''+u_\Om\ge 0$.
    \item The support function has some nice properties: 
    \begin{itemize}
        \item it is linear for the Minkowski sum and dilatation. Indeed, if $ \Om_1$ and $\Om_2$ are two convex bodies and $\alpha, \beta >0$, we have  
    $$h_{\alpha \Om_1 + \beta \Om_2}= \alpha h_{\Om_1}+ \beta h_{\Om_2},$$
    see \cite[Section 1.7.1]{schneider}.
        \item It also provides elegant formulas for some geometrical quantities. Fro example the perimeter of a convex body $\Om$ is given by 
        $$P(\Om) = \int_0^{2\pi}h_\Om(\theta)d\theta,$$
        and the Hausdorff distance between two convex bodies $\Om_1$ and $\Om_2$ is given by 
    $$d^H(\Om_1,\Om_2)=\sup_{\theta\in [0,2\pi)} |h_{\Om_1}(\theta)-h_{\Om_2}(\theta)|,$$
    see \cite[Lemma 1.8.14]{schneider}. 
    \end{itemize}
\end{itemize}
\end{remark}

At last, let us recall a classical result on the area of Minkowski sums. For more details and general results, we refer to \cite{schneider}. 
\begin{theorem}\label{th:brunn-minkowski}
There exist 3 bilinear (for Minkowski sum and dilatation) forms $W_k: \K^2\times \K^2\longrightarrow \mathbb{R}$, for $k\in  [\![0;2]\!]$, named Minkowski mixed volumes, such that for every  $K_1,K_2\in\K^2$  and $t_1,t_2>0$, we have
\begin{equation}\label{eq:mixed-volume}
|t_1K_1+t_2K_2|= t_1^2W_0(K_1,K_2)+2t_1t_2W_1(K_1,K_2)+t_2^2 W_2(K_1,K_2).
\end{equation}
Moreover, the $W_k$ are continuous with respect to the Hausdorff distance, in the sense that if two sequences of convex bodies $(K_1^n)_n$ and $(K_2^n)_n$ respectively converge to some convex bodies $K_1$ and $K_2$ both with respect to the Hausdorff distance, one has
$$\lim\limits_{n\rightarrow+\infty} W_k(K_1^n,K_2^n) = W_k(K_1,K_2).$$
\end{theorem}
}

\subsection{Preliminary lemmas}

In this section we prove some important Lemmas that we use in Section \ref{s:proofs} for the proofs of the main results. \vspace{2mm}

The following lemma\footnote{We are indebted to Jimmy Lamboley for the idea of proof of this lemma.} shows that the equality of Theorem \ref{th:KLR_polygon} which is valid for Cheeger-regular polygons becomes an inequality for general polygons. Thus, we obtain an upper bound for the Cheeger constant of polygons that we use to prove the inequality \eqref{eq:uper_bound}. 
\begin{lemma}\label{lem:jimmy}
If $\Om$ is a { convex} polygon, one has
$$h(\Om)\leq \frac{P(\Om)+\sqrt{P(\Om)^2-4\big(T(\Om)-\pi\big)|\Om|}}{2|\Om|},$$
{ where $T(\Omega)$ is defined in \eqref{eq:def_T}. The equality occurs if and only if the polygon $\Om$ is Cheeger-regular (i.e., all its sides touch its Cheeger set $C_\Om$). }

\end{lemma}
\begin{proof}

The number of sides of the inner sets $(\Om_{-t})_t$ of a polygon { (we refer to Definition \ref{def:inner} for the notion of inner sets)} is decreasing with respect to $t\ge 0$. Actually, the function $t\in[0,r(\Om)]\longmapsto n(t)$ (where $r(\Omega)$ is the inradius of $\Om$ and $n(t)$ is the number of sides of $\Om_{-t}$) is a piece-wise constant decreasing function. We introduce the sequence $0=t_0<t_1<\dots<t_{N_\Om}=r(\Om)$, where $N_\Om\in \mathbb{N}^*$, such that
$$\forall k\in\llbracket 0,N_\Om-1\rrbracket,\forall t\in [t_k,t_{k+1}),\ \ \ \ \ \ \ n(t) = n_k,$$
where $(n_k)_{k\in\llbracket 0,N_\Om-1\rrbracket}$ is strictly decreasing.\vspace{1mm}

{ In the computations below, we use the classical Steiner formulas (see \cite{steiner}) for the perimeter and the area of inner (polygonal) sets. We have for every $k\in\llbracket 1,N_\Om\rrbracket$,
\begin{equation}\label{eq:steiner_formula_1}
P(\Omega_{-t_{k}}) = P((\Omega_{-t_{k-1}})_{-(t_k-t_{k-1})})= P(\Om_{-t_{k-1}})-2(t_k-t_{k-1}) T(\Omega_{-t_{k-1}})   
\end{equation}
and 
\begin{equation}\label{eq:steiner_formula_2} |\Omega_{-t_{k}}|=|(\Omega_{-t_{k-1}})_{-(t_k-t_{k-1})}|=
|\Om_{-t_{k-1}}|-(t_k-t_{k-1}) P(\Om_{-t_{k-1}})+(t_k-t_{k-1})^2 T(\Om_{-t_{k-1}}).    
\end{equation}
} 

Let us take $t\in[0,r(\Om)]$ and $k\in \llbracket 1,N_\Om\rrbracket$. We have 
\begin{align*}
|\Om_{-t_k}|- (t-t_k)P(\Om_{-t_k})+(t-t_k)^2T(\Om_{-t_{k}})&= |\Om_{-t_{k-1}}|-(t_k-t_{k-1}) P(\Om_{-t_{k-1}})+(t_k-t_{k-1})^2T(\Om_{-t_{k-1}})\\
&- (t-t_k)\big( P(\Om_{-t_{k-1}})-2(t_k-t_{k-1})T(\Om_{-t_{k-1}})\big)+(t-t_k)^2 T(\Om_{-t_k})\\
&> |\Om_{-t_{k-1}}|-(t-t_{k-1})P(\Om_{-t_{k-1}})+(t-t_{k-1})^2T(\Om_{-t_{k-1}}), 
\end{align*}
where we used \eqref{eq:steiner_formula_1} and \eqref{eq:steiner_formula_2} for the equality and the fact that $T(\Om_{-t_{k-1}})<T(\Om_{-t_k})$ for the inequality (see the equation (14) of \cite{kawohl}). By straightforward induction, we show that one has for every $k\in \llbracket 1,N_\Om\rrbracket$,
\begin{equation}\label{eq:estimation_jimmy}
\forall t\in[0,r(\Om)],\ \ \ \ \ \ \ \  |\Om_{-t_k}|- (t-t_k)P(\Om_{-t_k})+(t-t_k)^2T(\Om_{-t_k})\ge |\Om|-tP(\Om)+t^2T(\Om).
\end{equation}

Now, let us take $k\in \llbracket 0,N_\Om-1\rrbracket$ and $t\in [t_k,t_{k+1})$. We have the inequality 
\begin{align*}
g(t):= |\Om_{-t}|-\pi t^2= |(\Om_{-t_k})_{-(t-t_k)}|-\pi t^2 &= |\Om_{-t_k}|- (t-t_k)P(\Om_{-t_k})+(t-t_k)^2T(\Om_{-t_k}) -\pi t^2\\
&\ge |\Om|-tP(\Om)+t^2T(\Om) -\pi t^2=:f(t),
\end{align*}
where the equality $g(t)=f(t)$ holds only on $[0,t_1]$. This { inequality yields}  that $1/h(\Om)$ { which is the  unique} zero of $g$ on $[0,r(\Om)]$ { (by Theorem \ref{th:KLR_convex})}, is larger than
$$\frac{2|\Om|}{P(\Om)+\sqrt{P(\Om)^2-4\big(T(\Om)-\pi\big)|\Om|}},$$
the smallest zero of $f$,\footnote{The second order polynomial function $f$ admits two positive zeros. Since, $f$ is continuous and $f(0) = |\Om|>0$ and $f(r(\Om))\leq g(r(\Om))=|\Om_{-r(\Om)}|-\pi r(\Om)^2 = -\pi r(\Om)^2<0$, we have by the Intermediate Value Theorem that the smallest zero of $f$ is in the interval $(0,r(\Om))$.} with equality if and only if the zero of $g$ is in $[0,t_1]$, which is the case if and only if the polygon $\Om$ is Cheeger-regular (see \cite[Theorem 3]{kawohl}). This ends the proof. 
\end{proof}
\begin{remark}
One interesting takeaway idea of proof of Lemma \ref{lem:jimmy} is that since the Cheeger constant of a planar convex set is { characterized via} the area of inner sets { (see Definition \ref{def:inner} and Theorem \ref{th:KLR_convex})}, one can use the estimates on the area of inner sets to obtain bounds on the Cheeger constant. This strategy is used in \cite{ftouhi_inequality} to prove the following sharp inequality 
$$\forall \Om \in \K^2,\ \ \ \ \ h(\Om)\leq \frac{1}{r(\Om)}+\sqrt{\frac{\pi}{|\Om|}},$$
where $r(\Om)$ is the inradius of $\Om$. 
\end{remark}

Since the inequality \eqref{eq:new_inequality} is quite easy to obtain for Cheeger-regular polygons (because in this case we have an explicit formula for the Cheeger constant in terms of the perimeter, the area and the interior angles), it is natural when dealing with general polygons to try to come back to the case of Cheeger-regular ones. The following Lemma shows how to deform a given polygon into a Cheeger-regular { one} while preserving its Cheeger constant, increasing its perimeter and decreasing its area. This allows (as shown in \textbf{Step 2} of Section \ref{ss:proof_of_new_inequality}) to prove the inequality \eqref{eq:new_inequality} for the case of general polygons. 

\begin{lemma}\label{lem:dorin}
Let $\Om$ be a polygon. There exists a Cheeger-regular polygon $\widetilde{\Om}$ such that: $|\Om|\ge |\widetilde{\Om}| $, $P(\Om) \leq  P\big(\widetilde{\Om}\big)$ and $h(\Om) = h\big(\widetilde{\Om}\big)$.
\end{lemma}
\begin{proof}
If $\Om$ is Cheeger-regular, we take $\widetilde{\Om}=\Om$. { Let us assume that the polygon $\Om$ is not Cheeger-regular}. We provide an algorithm to deform $\Om$ into a Cheeger-regular polygon with the same Cheeger set (thus, also the same Cheeger value), larger perimeter and smaller area. 

Since the polygon $\Om$ is not Cheeger-regular, there exist three consecutive vertices, that we denote { by} $X$, $Y$ and $Z$, such that at least one (or may be both) of the sides $[XY]$ and $[YZ]$ does not touch the Cheeger set $C_\Om$. 

\vspace{2mm}
\underline{\textbf{First step: using parallel chord movements}}
\vspace{2mm}

We begin by the case where both the sides { $[XY]$ and $[YZ]$} do not touch $C_\Om$. We use a parallel chord movement. More precisely, we move $Y$ along the line passing through $Y$ and being parallel to the line $(XZ)$. This way, the area is { preserved}, and the perimeter must increase when moving $Y$ away from the perpendicular bisector of the side $[XZ]$ (which is possible at least in one direction). We assume without loss of generality that the direction which increases the perimeter is from $Z$ to $X$ (see Figure \ref{fig:chord}). We then move $Y$ until one the following cases occurs:
\begin{enumerate}
\item the line $(XY)$ becomes colinear to the other side of extremity $X$.
\item The side $[YZ]$ touches the boundary of $C_\Om$. 
\end{enumerate}

\begin{figure}[h]
    \centering
    \includegraphics[scale=1]{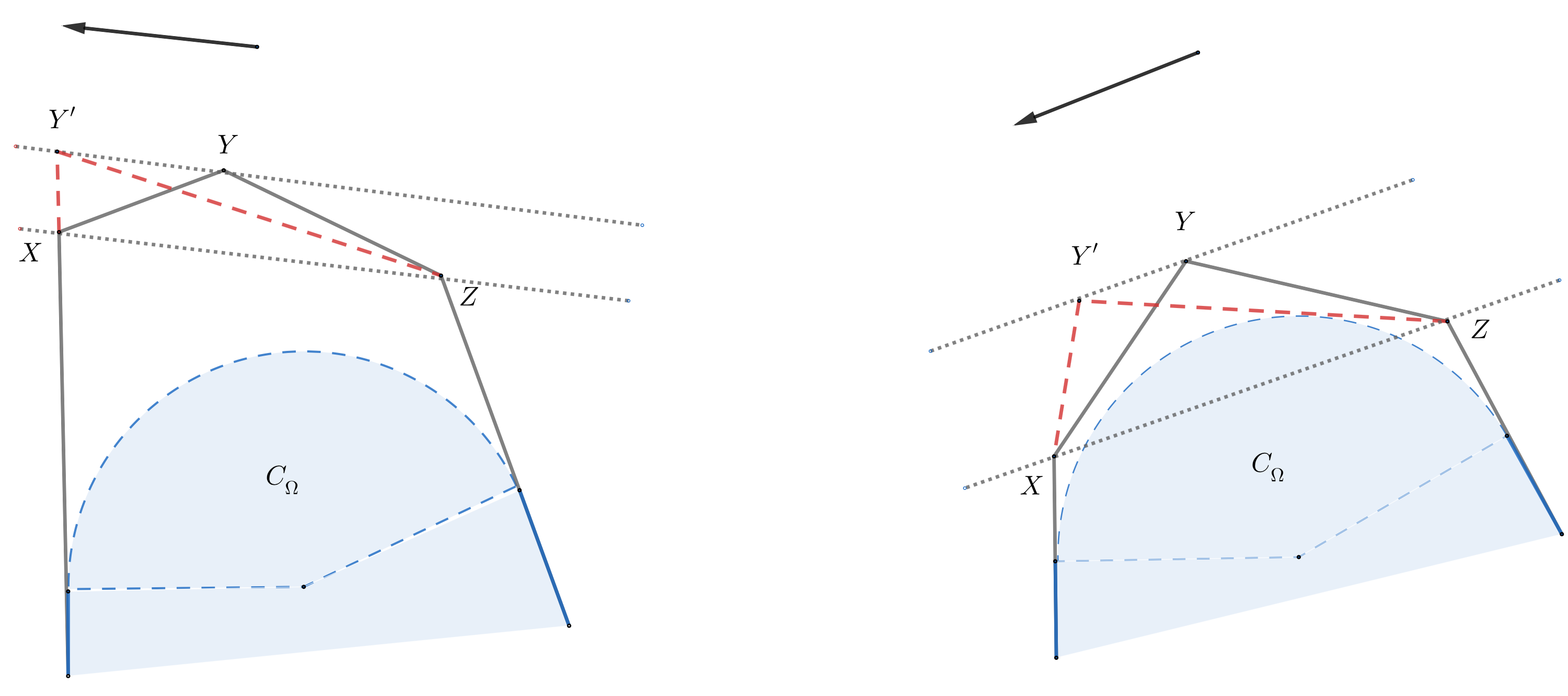}
    \caption{Case 1 on the left and case 2 on the right.}
    \label{fig:chord}
\end{figure}
In both cases, the number of sides that do not touch $\partial C_\Om$ { decreases} by one, while the area and the Cheeger constant are preserved and the perimeter is increased. 

We iterate { this} process for all the vertices which are extremities of two sides that do not touch $\partial C_\Om$. Since the number of vertices is finite, in a finite number of steps, we obtain a polygon where there are no consecutive sides that do not touch $\partial C_\Om$. 

\newpage 
\underline{\textbf{Second step: rotating the remaining sides}}\footnote{We are thankful to Dorin Bucur for suggesting such deformations.}
\vspace{2mm}

The second step is to "rotate" the remaining sides that do not touch $\partial C_\Om$ in such a way to make them touch it (see Figure \ref{fig:dorin}), in order to get a Cheeger-regular polygon with the same Cheeger constant, larger perimeter and smaller area. This kind of deformations {is inspired by the work of D. Bucur and I. Fragalà} \cite{bucur_fragala}. \vspace{2mm}

We denote by
\begin{itemize}
    \item $\alpha_1$,$\alpha_2\in (0,\pi)$ the interior angles of the polygon $\Om$ respectively associated to the vertices $X$ and $Y$,
    \item  $O$ the mid-point of the side $[XY]$,
    \item $t$ the angle of our "rotation",
    \item $X_t$ and  $Y_t$ the vertices of the { new polygon denoted by $\Om_t$, obtained by rotating the side $[XY]$ around the mid-point $O$ with the angle $t$.}
\end{itemize}
 { The polygon $\Omega_t$ has the same vertices as $\Om$ except $X$ and $Y$ that have respectively been mapped as $X_t$ and $Y_t$}, see Figure \ref{fig:dorin}.

\begin{figure}[h]
    \centering
    \includegraphics[scale=0.06]{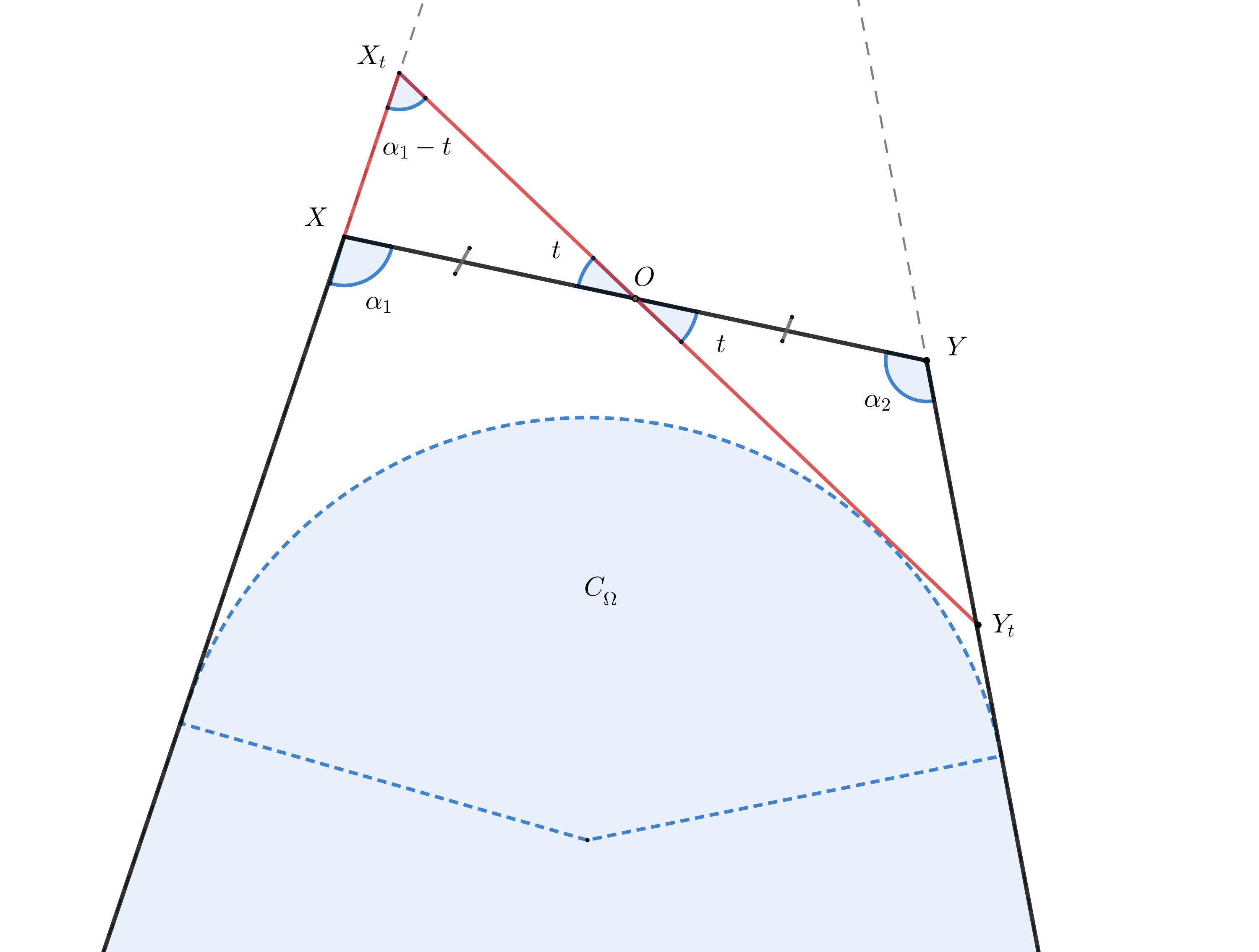}
    \caption{Rotation of the free side along its midpoint.}
    \label{fig:dorin}
\end{figure}

Without loss of generality, we assume that $\alpha_2\ge \alpha_1$ and we take $t \in (0,\pi-\alpha_2)$. It is trivial that $\Om$ and $\Om_t$ have the same Cheeger constant, moreover if the side $[XY]$ is not touching $\partial C_\Om$ then $\alpha_1+\alpha_2\ge \pi$ (see \cite[Section 5]{kawohl}). 

{ Let us now prove the inequalities $|\Omega_t|\leq |\Om|$ and $P(\Om_t)\ge P(\Om)$.} By using the sinus formula on the triangles $OXX_t$ and $OYY_t$, we have 
{ $$XX_t = a \frac{\sin t}{\sin(\alpha_1-t)}, \ \ \ \ \ \ \ \ OX_t = a \frac{\sin \alpha_1}{\sin(\alpha_1-t)}, \ \ \ \ \ \ \ \ YY_t = a \frac{\sin t}{\sin(\alpha_2+t)} \ \ \ \ \ \ \ \ \text{and} \ \ \ \ \ \ \ \ OY_t = a \frac{\sin \alpha_2}{\sin(\alpha_2+t)},$$}
where $a:=OX = XY/2$.
 \vspace{2mm}

Moreover, if we denote by $\S_{OXX_t}$ and $\S_{OYY_t}$ the areas of the triangles $OXX_t$ and $OYY_t$, we have
\begin{align*}
|\Om_t|- |\Om|&\ \ \ \ \ =\ \ \ \  \mathcal{S}_{OXX_t}-\mathcal{S}_{OYY_t}\\
&\ \ \ \ \ =\ \ \ \  \frac{1}{2}OX\cdot OX_t \sin t - \frac{1}{2}OY\cdot OY_t \sin t\\
&\ \ \ \ \ =\ \ \ \  \frac{a^2}{2}\left(\frac{\sin \alpha_1}{\sin(\alpha_1-t)}-\frac{\sin \alpha_2}{\sin(\alpha_2+t)}\right)\sin t\\
&\ \ \ \ \ =\ \ \ \  \frac{a^2\sin^2{(t)} \sin(\alpha_1+\alpha_2)}{2\sin(\alpha_1-t)\sin(\alpha_2+t)}\leq  0,
\end{align*}
{ because $\alpha_1+\alpha_2\in [\pi,2\pi]$, while $\alpha_1-t$ and $\alpha_2+t$ are in $(0,\pi]$.}

\vspace{2mm}

For the perimeters, we have
\begin{align*}
P(\Om_t)- P(\Om)&\ \ \ \ \ =\ \ \ \  XX_t+OX_t+OY_t-XY-YY_t\\
&\ \ \ \ \ =\ \ \ \  a\left(\frac{\sin t}{\sin(\alpha_1-t)}+\frac{\sin \alpha_1}{\sin(\alpha_1-t)}+ \frac{\sin \alpha_2}{\sin(\alpha_2+t)}-2-\frac{\sin t}{\sin(\alpha_2+t)}\right) \\
&\ \ \ \ \ =\ \ \ \  a\left(\frac{\sin t + \sin \alpha_1}{\sin(\alpha_1-t)}+\frac{\sin \alpha_2-\sin t}{\sin(\alpha_2+t)}-2\right) \\
&\ \ \ \ \ \ge\ \ \ \  a\left(\frac{\sin t + \sin \alpha_1}{\sin(\alpha_1-t)}+\frac{\sin \alpha_1-\sin t}{\sin(\alpha_1+t)}-2\right)\\
&\ \ \ \ \ =\ \ \ \  a\left(\frac{2\sin\left(\frac{\alpha_1+t}{2}\right)\cos\left(\frac{\alpha_1-t}{2}\right)}{2\sin\left(\frac{\alpha_1-t}{2}\right)\cos\left(\frac{\alpha_1-t}{2}\right)}+\frac{2\sin\left(\frac{\alpha_1-t}{2}\right)\cos\left(\frac{\alpha_1+t}{2}\right)}{2\sin\left(\frac{\alpha_1+t}{2}\right)\cos\left(\frac{\alpha_1+t}{2}\right)}-2\right) \\
&\ \ \ \ \ = \ \ \ \  a\left(\sqrt{\frac{\sin\left(\frac{\alpha_1+t}{2}\right)}{\sin\left(\frac{\alpha_1-t}{2}\right)}}-\sqrt{\frac{\sin\left(\frac{\alpha_1-t}{2}\right)}{\sin\left(\frac{\alpha_1+t}{2}\right)}}\  \right)^2\\
&\ \ \ \ \ \ge \ \ \ \ 0,
\end{align*}
where the inequality in the middle is a consequence of the assumption $\alpha_1\leq \alpha_2$ and the monotonicity of the function $g_t:x\longmapsto \frac{\sin{x}-\sin{t}}{\sin{(x+t)}}$ on $[\alpha_1,\alpha_2]$, where $t\in (0,\pi-\alpha_2)$. Indeed, we have for every $x\in [\alpha_1,\alpha_2]$
\begin{align*}
g'_t(x)&\ \ \ \ \ =\ \ \ \   \frac{\cos{x}\sin{(x+t)}-(\sin{x}-\sin{t})\cos{(x+t)}}{\sin^2(x+t)}\\
&\ \ \ \ \ =\ \ \ \  \frac{(\cos{x}\sin{(x+t)}-\sin{x}\cos{(x+t)})+\sin{t}\cos{(x+t)}}{\sin^2(x+t)}\\
&\ \ \ \ \ =\ \ \ \  \frac{\sin{t}+\sin{t}\cos{(x+t)}}{\sin^2(x+t)}\\
&\ \ \ \ \ =\ \ \ \ \frac{\sin{t}\cdot (1+\cos{(x+t)})}{\sin^2(x+t)}\\
&\ \ \ \ \ \ge\ \ \ \ 0.
\end{align*}

Iterating, in a finite number of steps, we get a Cheeger-regular polygon $\tilde{\Om}$ such that $|\Om|\ge |\widetilde{\Om}| $, $P(\Om) \leq  P\big(\widetilde{\Om}\big)$ and $h(\Om) = h\big(\widetilde{\Om}\big)$.

\end{proof}


In the following lemmas, we give some quantitative estimates for { the perimeter,} the Cheeger constant and the area via the Hausdorff distance and the radial functions (see Definition \ref{def:parametrizations}). These inequalities are used in the fourth step of Section \ref{sect:preuve_convex}. 

\begin{lemma}\label{lem:quantitative_p}
If $\Om_1$ and $\Om_2$ are two planar convex sets, we have 
\begin{equation}\label{eq:quantitative_p}
|P(\Om_1)-P(\Om_2)|\leq 2\pi d^H(\Om_1,\Om_2).    
\end{equation}
\end{lemma}
\begin{proof}
Let us respectively denote by $h_{\Om_1}$ and $h_{\Om_2}$ the support functions of $\Om_1$ and $\Om_2$, see Definition \ref{def:parametrizations}. We have 
$$|P(\Om_1)-P(\Om_2)| = \left|\int_0^{2\pi} h_{\Om_1}-\int_0^{2\pi} h_{\Om_2}\right|
\leq 2\pi \left\| h_{\Om_1}-h_{\Om_2}\right\|_{\infty}
= 2\pi d^H(\Om_1,\Om_2),$$
where $\|\cdot\|_\infty$ stands for the infinity norm over $[0,2\pi)$.
\end{proof}

\begin{lemma}\label{lem:quantitative}
Take $\Om_1$ and $\Om_2$ two planar domains starshaped with respect to the origin $0$, whose radial functions are denoted by $f_{\Om_1}$ and $f_{\Om_2}$ and such that $f_{\Om_1},f_{\Om_2}\ge r_0$, where $r_0>0$. We have
\begin{enumerate}
\item $|h(\Om_1)-h(\Om_2)|\leq \frac{2}{r_0^2} \|f_{\Om_1}-f_{\Om_2}\|_{\infty},$
\item $\big||\Om_1|-|\Om_2|\big|\leq 2\pi \max(\|f_{\Om_1}\|_\infty,\|f_{\Om_2}\|_\infty) \|f_{\Om_1}-f_{\Om_2}\|_\infty$,
\end{enumerate}
where $\|\cdot\|_\infty$ stands for the infinity norm over $[0,2\pi)$. 
\end{lemma}
\begin{proof}

\begin{enumerate}
\item The proof of this assertion is inspired from the proof of \cite[Proposition 1]{cox}.

Let us take $d:= \|f_{\Om_1}-f_{\Om_2}\|_\infty$. We have for every $\theta\in [0,2\pi)$
$$\left(1+\frac{d}{r_0}\right)f_{\Om_1}(\theta) = f_{\Om_1}(\theta)+\frac{d}{r_0}f_{\Om_1}(\theta)\ge f_{\Om_1}(\theta) + d \ge f_{\Om_1}(\theta)+f_{\Om_2}(\theta)-f_{\Om_1}(\theta)=f_{\Om_2}(\theta).$$
Thus, $\Om_2\subset \left(1+\frac{d}{r_0}\right)\Om_1$, which implies the following 
$$h(\Om_1)\leq h\left(\frac{1}{1+d/r_0}\Om_2\right) = \left(1+\frac{d}{r_0}\right) h(\Om_2)\leq h(\Om_2) + \frac{d}{r_0} h(B_{r_0}) = h(\Om_2)+\frac{2d}{r_0^2},$$
where $B_{r_0}$ is the disk of radius $r_0$. 
By similar arguments we obtain
$$h(\Om_2)\leq h(\Om_1)+\frac{2d}{r_0^2},$$
which proves the announced inequality.
\item We have \begin{align*}
\big||\Om_1|- |\Om_2|\big|&= \frac{1}{2} \left|\int_0^{2\pi}(f_{\Om_1}^2(\theta)-f_{\Om_2}^2(\theta))d\theta\right| \\
&\leq  \frac{1}{2} \int_0^{2\pi} (|f_{\Om_1}(\theta)|+|f_{\Om_2}(\theta)|) |f_{\Om_1}(\theta)-f_{\Om_2}(\theta)|d\theta\\
&\leq  2\pi\max(\|f_{\Om_1}\|_\infty,\|f_{\Om_2}\|_\infty) \|f_{\Om_1}-f_{\Om_2}\|_\infty.
\end{align*}
\end{enumerate}

\end{proof}

\section{{ Proofs} of the main results}\label{s:proofs}
\subsection{Proof of the inequality \eqref{eq:new_inequality}}\label{ss:proof_of_new_inequality}
The proof is done in four steps: \vspace{2mm}

\underline{\textbf{Step 1: Cheeger-regular polygons}}\vspace{2mm}

Even-though the inequality was already known in this case, we briefly recall the proof for { the} sake of completeness. 

Since $\Om$ is a Cheeger-regular polygon, by Theorem \ref{th:KLR_polygon}, we { have} an explicit formula of its Cheeger constant. We then use the inequality $T(\Om)\leq \frac{P(\Om)^2}{4|\Om|}$ of  Theorem \ref{th:T_estimates} to conclude (see \eqref{eq:def_T} for the definition of $T(\Om)$).  We write 
$$h(\Om) = \frac{P(\Om)+\sqrt{P(\Om)^2-4\big(T(\Om)-\pi\big)|\Om|}}{2|\Om|}\ge \frac{P(\Om)+\sqrt{P(\Om)^2-4\left(\frac{P(\Om)^2}{4|\Om|}-\pi\right)|\Om|}}{2|\Om|}= \frac{P(\Om)+\sqrt{4\pi|\Om|}}{2|\Om|}.$$

\underline{\textbf{Step 2: General polygons}}\vspace{2mm}

By Lemma \ref{lem:dorin}, there exists $\widetilde{\Om}$ a Cheeger-regular polygon  such that { $|\Om|\ge |\widetilde{\Om}| $, $P(\Om) \leq P\big(\widetilde{\Om}\big)$} and $h(\Om) = h\big(\widetilde{\Om}\big)$. 

Then, we get
$$h(\Om) = h\big(\widetilde{\Om}\big) \ge  \frac{P\left(\widetilde{\Om}\right)+\sqrt{4\pi\left|\widetilde{\Om}\right|}}{2\left|\widetilde{\Om}\right|} = \frac{P\big(\widetilde{\Om}\big)}{2\left|\widetilde{\Om}\right|}+\frac{\pi}{\sqrt{\left|\widetilde{\Om}\right|}} \ge \frac{P(\Om)}{2|\Om|}+\frac{\pi}{\sqrt{|\Om|}} = \frac{P(\Om)+\sqrt{4\pi|\Om|}}{2|\Om|}.$$

\underline{\textbf{Step 3: General convex sets}}\vspace{2mm}

By { the} density of the polygons in $\K^2$ and the continuity of the area, the perimeter and the Cheeger constant { with respect to} the Hausdorff distance, we show that the inequality \eqref{eq:new_inequality} holds for general convex sets. { We refer to \cite[Proposition 3.1]{reverse_cheeger} for the continuity of the Cheeger constant with respect to the Hausdorff distance in the class of convex sets.}
\vspace{2mm}

\underline{\textbf{Step 4: Equality for sets that are homothetical to their form bodies}}\vspace{2mm}

If $\Om$ is homothetical to its form body (which is  the case for circumscribed polygons), we have by using \cite[Equality (7.168)]{schneider} and the equality $\frac{1}{2}r(\Om)P(\Om)=|\Om|$ { (see for example \cite{circum})} 
$$\forall t\in [0,r(\Om)],\ \ \ |\Om_{-t}|=\left(1-\frac{t}{r(\Om)}\right)^2|\Om|=|\Om|-P(\Om)t+\frac{P(\Om)^2}{4|\Om|}t^2.$$ 
{ The equation $|\Om_{-t}|=\pi t^2$ admits two different solutions $\frac{2|\Om|}{P(\Om)-\sqrt{4\pi|\Om|}}$ and $\frac{2|\Om|}{P(\Om)+\sqrt{4\pi|\Om|}}$. Thus, by Theorem \ref{th:KLR_convex} we have 
$$h(\Om)\in \left\{\frac{P(\Om)-\sqrt{4\pi|\Om|}}{2|\Om|}, \frac{P(\Om)+\sqrt{4\pi|\Om|}}{2|\Om|} \right\}.$$
At last, since $h(\Om)\ge \frac{P(\Om)+\sqrt{4\pi|\Om|}}{2|\Om|}>\frac{P(\Om)-\sqrt{4\pi|\Om|}}{2|\Om|}$ (by inequality \eqref{eq:new_inequality}), we deduce the equality }
\begin{equation}\label{eq:egal}
h(\Om)= \frac{P(\Om)+\sqrt{4\pi|\Om|}}{2|\Om|}.    
\end{equation}

\subsection{Proof of the second assertion of Theorem \ref{th:diagram_open} (convex sets)}\label{sect:preuve_convex}

The inequalities \eqref{eq:trivial} and  \eqref{eq:new_inequality} (stated in the introduction) imply that 
$$\D_{\mathcal{K}^2}\subset \left\{(x,y)\ \ |\ \ x\ge x_0\ \ \ \text{and}\ \ \  \frac{1}{2}x+\sqrt{\pi}\leq y \leq x \right\}.$$

It remains to prove the { opposite} inclusion. The proof follows the following steps:

\begin{enumerate}
\item  We provide a continuous family $(S_p)_{p\ge P(B)}$ of convex bodies which fill the upper boundary of the diagram.
\item We provide a continuous family $(L_p)_{p\ge P(B)}$ of convex bodies which fill the lower boundary of the diagram.
\item We use the latter domains to construct (via Minkowski sums) a family of continuous paths $(\Gamma_p)_{p\ge P(B)}$ which relate the upper domains to the lower ones. By continuously increasing the perimeter, we show that we are able to cover all the area between the upper and lower boundaries of the diagram $\D_{\K^2}$ (see Definition \ref{def:boundary} for the notion of upper and lower boundaries of the diagram $\D_{\K^2}$). 
\end{enumerate}\vspace{2mm}

\underline{\textbf{Step 1: The upper boundary of the diagram $\D_{\K^2}$:}}\vspace{2mm}

The upper boundary of $\D_{\K^2}$ (see Definition \ref{def:boundary}) is filled by domains that are Cheeger of themselves, which means that $C_\Om = \Om$. It is shown in \cite[Theorem 2]{kawohl} that the stadiums  (i.e., the convex hull of two identical disks) are Cheeger of themselves. We then use these sets to fill the upper boundary $\{(x,x)\ |\ x\ge P(B)\}$.

Let us consider the family of stadiums $(Q_t)_{t\ge 0}$ given by the convex hulls of the balls of unit radius centered in $O(0,0)$ and $O_t(0,t)$ rescaled so as $|Q_t|=1$. The function $t\in[0,+\infty) \longmapsto P(Q_t)=\frac{2(\pi+t)}{\sqrt{\pi+2t}} $ is continuous and strictly increasing to infinity. Thus, we have by the { Intermediate Value Theorem}

$$\big\{\big(P(Q_t),h(Q_t)\big)\ |\ t\ge 0\}=\big\{\big(P(Q_t),P(Q_t)\big)\ |\ t\ge 0\}=\{(x,x)\ |\ x\ge P(B)\}.$$


\vspace{3mm}
\underline{\textbf{Step 2: The lower boundary of the diagram $\D_{\K^2}$:}}\vspace{2mm}

Since the equality \eqref{eq:egal} holds for sets that are homothetical to their form bodies, we use such domains to fill the lower boundary. { Let us consider the family $(C_d)_{d\ge 2}$ of the so-called symmetrical cup-bodies, which are given by the convex hulls of the unit ball (centered in $O(0,0)$) and the points of coordinates $(-d/2,0)$ and $(d/2,0)$ rescaled so as $|C_d|=1$. By using formulas (7) and (8) of \cite{cifre_maria}, we have for every $d\ge 2$ 
$$P(C_d) = 2 \sqrt{\sqrt{d^2-4}+2\arcsin{\left(\frac{2}{d}\right)}}.$$
The function $d\in [2,+\infty)\longmapsto P(C_d) = 2 \sqrt{\sqrt{d^2-4}+2\arcsin{\left(\frac{2}{d}\right)}}$ is continuous and strictly increasing to infinity. Thus, we have by the Intermediate Value Theorem
$$\big\{\big(P(C_d),h(C_d)\big)\ |\ d\ge 2\big\}=\big\{\big(P(C_d),P(C_d)/2+\sqrt{\pi}\big)\ |\ d\ge 2\}=\{(x,x/2+\sqrt{\pi})\ |\ x\ge P(B)\}.$$}


\vspace{3mm}
\underline{\textbf{Step 3: Continuous paths:}}\vspace{2mm}

Since the functions $t\in[0,+\infty) \longmapsto P(Q_t)=\frac{2(\pi+t)}{\sqrt{\pi+2t}} $ and $d\in [2,+\infty)\longmapsto P(C_d) = 2\sqrt{\sqrt{d^2-4}+2\arcsin{\frac{2}{d}}}$ are continuous and strictly increasing, we have 
$$\forall p\ge P(B),\exists ! (t_p,d_p)\in [0,+\infty)\times [2,+\infty),\ \ \  P(Q_{t_p})=P(C_{d_p})=p.$$
From now on we take $S_p:=Q_{t_p}$ and $L_p:=C_{d_p}$. 
\vspace{2mm}

For every $p\ge P(B)$, we introduce the closed and continuous path $\Gamma_{p}:[0,3)\longrightarrow \R^2$, defined as follows:
$$\begin{array}{ccccc}
 t & \longmapsto & \begin{cases}
    \big(P(K_p^{t}),h(K_p^{t})\big), &\qquad\mbox{if $t\in [0,1],$}\vspace{1mm} \\
    \big((t-1)P(B)+(2-t)p,(t-1)P(B)+(2-t)p\big),& \qquad\mbox{if $t\in (1,2],$} \vspace{1mm} \\
     \big((3-t)P(B)+(t-2)p,\ (3-t)P(B)+(t-2)(p/2+\sqrt{\pi}))\big),& \qquad\mbox{if $t\in (2,3),$}
\end{cases}\\
\end{array}$$
{ with $$K_p^t := \frac{t S_p\oplus(1-t)L_p}{\sqrt{|t S_p\oplus(1-t)L_p|}}\in \K^2,$$
where $t S_p\oplus(1-t)L_p$ is the Minkowski sum of the sets $t S_p$ and $(1-t) L_p$ (see Definition \ref{def:minkowski_sums}).}

\begin{center}
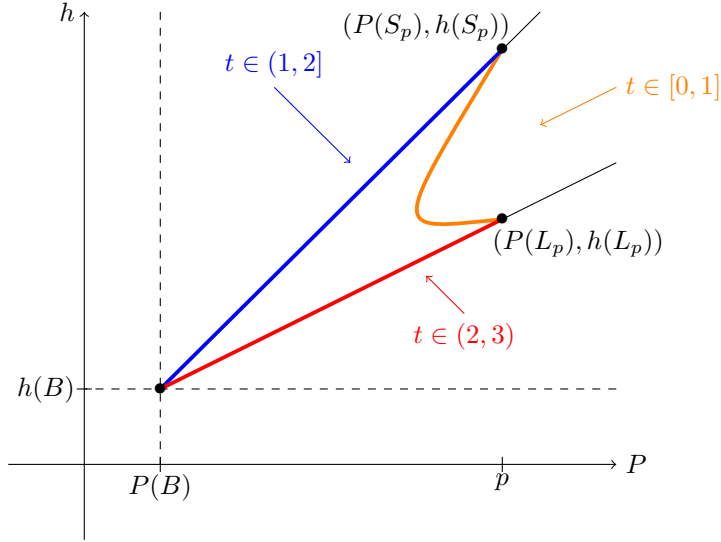

\begin{tikzpicture}
\draw[->] (-1,0) -- (7,0);
\draw (7,0) node[right] {$P$};
\draw [->] (0,-1) -- (0,6);
\draw (0,6) node[left] {$h$};
\draw [dashed] (1,6) -- (1,0) node[below] {$P(B)$};
\draw [dashed] (7,1) -- (0,1) node[left] {$h(B)$};
\draw [domain=1:7] plot(\x,{.5*\x+.5});
\draw [domain=1:6] plot(\x,{\x});
\draw (5.5,0) node[below] {$p$};
\draw (5.5,-0.1) -- (5.5,0.1);
\draw (1,-0.1) -- (1,0.1);
\draw (-0.1,1) -- (0.1,1);
\draw [line width = .5mm, orange]  (5.5,3.25) .. controls (4,3.1) and (4,3)   .. (5.5,5.5);
\draw [line width = .5mm, blue] (1,1) -- (5.5,5.5);
\draw [line width = .5mm, red] (1,1) -- (5.5,3.25);
\draw (5.5,5.5) node {$\bullet$};
\draw (4.5,5.5) node[above] {$(P(S_p),h(S_p))$};
\draw (5.5,3.25) node {$\bullet$};
\draw (6.5,3.25) node[below] {$(P(L_p),h(L_p))$};
\draw (1,1) node {$\bullet$};
\draw[->,blue] (2.5,5) -- (3.5,4);
\draw [blue] (2.5,5) node[above] {$t\in(1,2]$};
\draw[->,red] (5,2) -- (4.5,2.5);
\draw [red] (5,2) node[below] {$t\in(2,3)$};
\draw[->,orange] (7,5) -- (6,4.5);
\draw [orange] (7,5) node[right] {$t\in[0,1]$};
\end{tikzpicture}
\captionof{figure}{The continuous and closed path $\Gamma_p$}
\label{fig:path}
\end{center}

{ For every $t\in [0,1]$, the set $K^t_p$ is convex with unit area. Moreover, by the continuity of the perimeter, the area and the Cheeger constant with respect to the Hausdorff distance, we have that the set $\{(P(K^t_p),h(K^t_p))\ |\ t\in [0,1]\}$ is a continuous curve in $\R^2$. Thus, we conclude that the path $\Gamma_p$ is a closed and continuous curve in $\R^2$. }


Since the diameters of $L_p$ and $S_p$ are colinear, we can use the following result of \textbf{Step 3} of the proof of { \cite[Theorem 3.14]{ftouh}}:
\begin{equation}\label{eq:estimations_peri}
\forall t\in [0,1],\ \ \ \ \ \frac{p}{2} \leq P(K_p^t).
\end{equation}

\vspace{2mm}

\underline{\textbf{Step 4: Stability of the paths:}}\vspace{2mm}

Now, let us prove a continuity result on the paths $\big(\Gamma_p\big)_{p\ge P(B)}$. We take $p_0\ge P(B)$ and $\eps>0$, and show that 
\begin{equation}\label{eq:convergence}
\exists\ \alpha_\eps>0, \forall p\in (p_0-\alpha_\eps,p_0+\alpha_\eps)\cap [P(B),+\infty), \ \ \ \ \ \ \ \ \ \sup_{t\in [0,3]}\|\ \Gamma_p(t)-\Gamma_{p_0}(t)\ \|\leq \eps.
\end{equation}
Let $p\in [P(B),p_0+1]$, with straightforward computations, we have that for every $t\in[1,3)$,
$$\|\Gamma_p(t)-\Gamma_{p_0}(t)\|\leq 2|p-p_0|\underset{p\rightarrow p_0}{\longrightarrow} 0.$$
The remaining case ($t\in [0,1]$) requires more computations. For every $t\in[0,1]$, we have
$$\|\Gamma_p(t)-\Gamma_{p_0}(t)\|\leq |P(K^t_p)-P(K^t_{p_0})|+|h(K^t_p)-h(K^t_{p_0})|\leq \underbrace{\left(2\pi +\frac{(p_0+1)^6}{2}\right)}_{C_{p_0}>0}d^H(K^t_p,K^t_{p_0}).$$

Indeed, we used:
\begin{itemize}
\item Lemma \ref{lem:quantitative_p} for the term with the perimeters 
$$|P(K^t_p)-P(K^t_{p_0})| \leq  2\pi d^H(K^t_p,K^t_{p_0}),$$

\item and the first assertion of Lemma \ref{lem:quantitative} for the term with the Cheeger constants, with the sets $K^t_p$ and $K^t_{p_0}$ that we assume to contain the origin $0$ and whose radial functions (see Definition \ref{def:parametrizations}) are denoted by $f_{p,t},f_{p_0,t}$.
\begin{align*}
|h(K^t_p)-h(K^t_{p_0})| &\leq  \frac{2}{\min\big(r(K^t_p),r(K^t_{p_0})\big)^2} \cdot \|f_{p,t}-f_{p_0,t}\|_\infty\ \ \ \ \text{(by Lemma \ref{lem:quantitative})}   \\
& \leq \frac{2}{\min\big(r(K^t_p),r(K^t_{p_0})\big)^2}\cdot \frac{\|f_{p,t}\|_\infty \|f_{p_0,t}\|_\infty}{\min\big(r(K^t_p),r(K^t_{p_0})\big)^2}\cdot   d^H(K^t_p,K^t_{p_0}) \ \ \ \ \text{(by \cite[Proposition 2]{boulkhemair})} \\
& \leq \frac{(p_0+1)^6}{2} d^H(K^t_p,K^t_{p_0}) \ \ \ \text{(we used $\|f_\Om\|_\infty\leq d(\Om)\leq \frac{P(\Om)}{2}$ and $r(\Om)> \frac{|\Om|}{P(\Om)}$, see \cite[Lemma B.1]{brasco}).}
\end{align*}
\end{itemize}

{ It remains to prove that $d^H(K_{p}^t,K_{p_0}^t)$ converges (uniformly in $t$) to $0$ when $p$ goes to $p_0$. Before detailing the computations, let us recall that if $\Om$ is a convex body, we denote by $h_\Om$ its support function defined in Definition \ref{def:parametrizations} (we refer the reader to Remark \ref{rk:parametrizations} for some interesting properties of support functions). We also note that if $\Om_1, \Om_2\in \K^2$ such that $|\Om_1|=|\Om_2|=1$, one has
$$\forall t\in [0,1],\ \ \ \ |(1-t)\Om_1\oplus t\Om_2|\ge ((1-t)|\Om_1|^{1/2}+t|\Om_2|^{1/2})^2=1,$$
where we used the classical Brunn--Minkowski inequality (see \cite[Theorem 7.1.1]{schneider} for example). We are now in position to conclude.} 

\begin{eqnarray*}
d^H(K_{p}^t,K_{p_0}^t) &=&\left \|h_{K_{p}^t}-h_{K_{p_0}^t}\right \|_{\infty}\ \ \ \text{(by \cite[Lemma 1.8.14]{schneider})}\\
 &=& \left \|\frac{(1-t)h_{L_{p_0}}+ t h_{S_{p_0}}}{\sqrt{|(1-t){L_{p_0}}\oplus t {S_{p_0}}|}}-\frac{(1-t)h_{L_p}+t h_{S_p}}{\sqrt{|(1-t){L_p}\oplus t {S_p}|}}\right \|_{\infty}\\
 &\leq &(1-t) \left \|\frac{h_{L_{p_0}}}{\sqrt{|(1-t){L_{p_0}}\oplus t {S_{p_0}}|}}-\frac{h_{L_p}}{\sqrt{|(1-t){L_p}\oplus t {S_p}|}}  \right \|_{\infty}\\
&&+t \left \|\frac{h_{S_{p_0}}}{\sqrt{|(1-t){L_{p_0}}\oplus t {S_{p_0}}}|}-\frac{h_{S_p}}{\sqrt{|(1-t){L_p}\oplus t {S_p}|}}  \right \|_{\infty}\\
&\leq &\frac{1}{\sqrt{|(1-t){L_p}\oplus t {S_p}|}}\left(\left \|h_{S_{p_0}}-h_{S_p}\right \|_{\infty}+\left \|h_{L_{p_0}}-h_{L_p}\right \|_{\infty}\right)\\
&& + \left( \left \|h_{S_{p_0}}\right \|_{\infty}+\left \|h_{L_{p_0}}\right \|_{\infty}\right)\left |\frac{1}{\sqrt{|(1-t){L_p}\oplus t {S_p}|}}-\frac{1}{\sqrt{|(1-t){L_{p_0}}\oplus t {S_{p_0}}|}} \right |\\
&\leq & d^H(S_{p_0},S_p) +d^H(L_{p_0},L_p) +\left( \left \|h_{S_{p_0}}\right \|_{\infty}+\left \|h_{L_{p_0}}\right \|_{\infty}\right)\Big |\ |(1-t)L_p\oplus t S_p|-|(1-t)L_{p_0}\oplus t S_{p_0}|\ \Big| \\
&\leq & d^H(S_{p_0},S_p) +d^H(L_{p_0},L_p)+ \left( \left \|h_{S_{p_0}}\right \|_{\infty}+\left \|h_{L_{p_0}}\right \|_{\infty}\right) \sum_{k=0}^{2}|W_k(L_p,S_p)-W_k(L_{p_0},S_{p_0})| \underset{p\rightarrow p_0}{\longrightarrow} 0,
\end{eqnarray*}
{ where $W_0$, $W_1$ and $W_2$ stand for the Minkowski mixed volumes introduced in Theorem \ref{th:brunn-minkowski}.} \vspace{2mm}

Finally, we deduce that $\lim\limits_{p\rightarrow p_0} \sup\limits_{t\in [0,3]}\|\ \Gamma_p(t)-\Gamma_{p_0}(t)\|=0$, which proves \eqref{eq:convergence}.

\vspace{3mm}
\underline{\textbf{Step 5: Conclusion:}}\vspace{2mm}

{ Now that we proved that the boundaries $\{(x,x)\ \ |\ \ x\ge P(B)\}$ and $\{(x,x/2+\sqrt{\pi})\ \ |\ \ x\ge P(B)\}$ are included in the diagram $\D_{\K^2}$, it remains to show that it is also the case for the set of points contained between them. Let  $A(x_A,y_A)\in \left\{(x,y)\ \ |\ \ x> x_0\ \ \ \text{and}\ \ \  x/2+\sqrt{\pi}< y < x \right\}$. Step 4 shows that for any choice of $p_0$ and $p_1$ in $[P(B),+\infty)$, the curves $\Gamma_{p_0}$ and $\Gamma_{p_1}$ are homotopic, and the homotopy is $H:(\sigma,t)\in [0,1]\times [0,3)\longmapsto \Gamma_{(1-\sigma)p_0+\sigma p_1}$. In particular, let us chose $p_0=P(B)$ and $p_1=4 x_A$. The index of $A$ with respect to $\Gamma_{P(B)}=\{(P(B),P(B))\}$ is equal to $0$. Meanwhile, by the inequality \eqref{eq:estimations_peri} of Step 3, we deduce that $A$ is in the interior of the curve $\Gamma_{4 x_A}$, which means that its index with respect to $\Gamma_{4 x_A}$ is non-zero. Thus, it must follow that there exists $(\Bar{\sigma},\Bar{t})\in [0,1]\times [0,3)$ such that $A=H(\Bar{\sigma},\Bar{t})\in \D_{\K^2}$.
}

Finally, we get the equality 
$$\D_{\mathcal{K}^2}=\left\{(x,y)\ \ |\ \ x\ge x_0\ \ \ \text{and}\ \ \  \frac{1}{2}x+\sqrt{\pi}\leq y \leq x \right\}.$$

\subsection{Proof of the first assertion of Theorem  \ref{th:diagram_open} (simply connected sets)}
By the inequalities \eqref{eq:trivial} and \eqref{eq:faber-krahn} where the latter one is an equality if and only if $\Om$ is a ball,  we have 
$$\D_{\mathcal{S}^2}\subset \{(x_0,x_0)\}\cup\left\{(x,y)\ \ |\ \ x> x_0\ \ \ \text{and}\ \ \ x_0<y\leq x\right\}.$$

We have $(x_0,x_0) = \big(P(B),h(B)\big)\in \D_{\mathcal{S}^2}$. Take $(p,\ell)\in \left\{(x,y)\ \ |\ \ x>x_0\ \ \ \text{and}\ \ \ x_0<y\leq x\right\}$, let us prove that there exists a simply connected domain $\Om\subset \R^2$ of unit area such that $P(\Om)=p$ and $h(\Om)=\ell$. \vspace{1mm}

If $\ell \ge p/2+\sqrt{\pi}$, then by the second assertion of Theorem \ref{th:diagram_open} there exists a convex (thus simply connected) domain satisfying the latter properties. Now, let us assume that $\ell < p/2+\sqrt{\pi}$. We take { $L_{2(\ell-\sqrt{\pi})}$ as in the proof of the second assertion of Theorem \ref{th:diagram_open} (see Step 3 of Section \ref{sect:preuve_convex}) to be a symmetrical 2-cup body (which is the convex hull of a disk and two points outside it that are symmetric with respect to its center) such that $P(L_{2(\ell-\sqrt{\pi})})=2(\ell-\sqrt{\pi})<p$, $|L_{2(\ell-\sqrt{\pi})}|=1$ and $h(L_{2(\ell-\sqrt{\pi})})=\ell$ (where we used the equality \eqref{eq:egal} since $L_{2(\ell-\sqrt{\pi})}$ is homothetical to its form body). Since, the involved functionals are invariant with respect to translations and rotations, we may assume without loss of generality that $L_{2(\ell-\sqrt{\pi})}$ is symmetric with respect to the x-axis, included  $ \{x\leq 0\}$ and its boundary touches the y-axis in the origin $0$ which is assumed to be a singular point, see Figure \ref{fig:tailed}. Let $\eps>0$ sufficiently small such that the set $C_\ell$ (the Cheeger set of $L_{2(\ell-\sqrt{\pi})}$) is included in the half-plane $\{x<-\eps\}$. We denote by $A(-\eps,c_\eps)$ and $B(-\eps,-c_\eps)$, where $c_\eps>0$, the points of the intersection between the line $\{x=-\eps\}$ and the boundary of $L_{2(\ell-\sqrt{\pi})}$ (see Figure \ref{fig:tailed}). For $t\ge 0$, we introduce the points $A_t\left(-\eps,\frac{\eps c_\eps}{\eps+t}\right)$, $B_t\left(-\eps,-\frac{\eps c_\eps}{\eps+t}\right)$ and $O_t\left(t,0\right)$. We then define for every $t\ge 0$, 
$$L^t:=(L_{2(\ell-\sqrt{\pi})}\cap \{x\leq-\eps\})\cup \mathcal{T}_t,$$
where $\mathcal{T}_t$ is the (closed) triangle of vertices $O_t$, $A_t$ and $B_t$. The function $t\ge 0\longmapsto P(L^t)$ continuously varies from $2(\ell-\sqrt{\pi})$ to $+\infty$. Thus, by the Intermediate Value Theorem, there exists $t_p$ such that $P(L^{t_p})=p$. Moreover, the set $L^{t_p}$ is simply connected with unit area and has the same Cheeger set as $L_{2(\ell-\sqrt{\pi})}$, which yields that $h(L^{t_p})=h(L_{2(\ell-\sqrt{\pi})})=\ell$. This shows that $(p,\ell)\in \D_{\mathcal{S}^2}$. }


Finally, we obtain the equality
$$\D_{\mathcal{S}^2}=\{(x_0,x_0)\}\cup \left\{(x,y)\ \ |\ \ x> x_0\ \ \ \text{and}\ \ \ x_0<y\leq x\right\}.$$

\begin{figure}[h]
    \centering
    \includegraphics[scale=0.7]{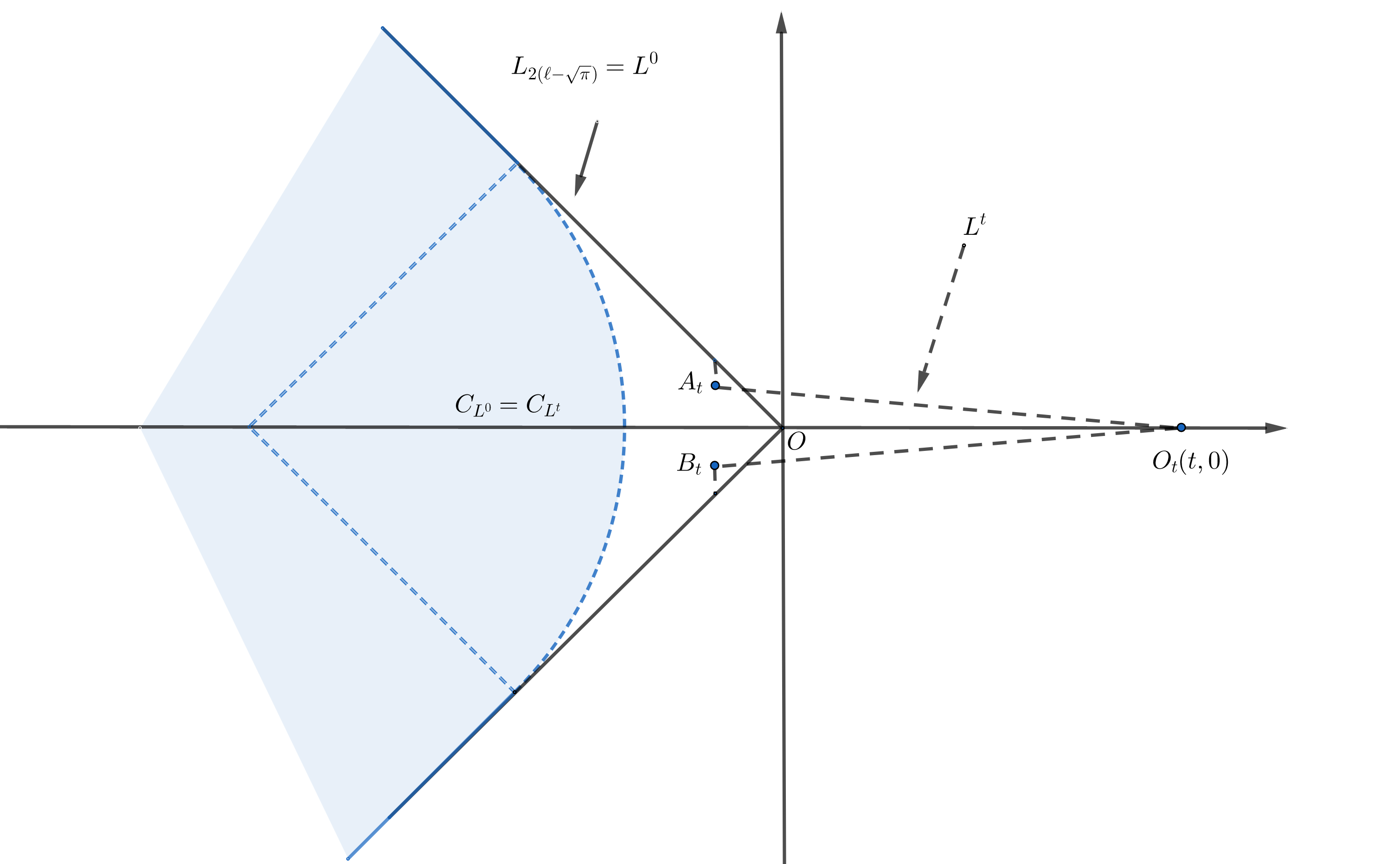}
    \caption{Tailed domain $L^t$ with the same area, the same Cheeger set and higher perimeter.}
    \label{fig:tailed}
\end{figure}

\subsection{Proof of the inequality \eqref{eq:uper_bound}}
This is a direct application of Lemma \ref{lem:jimmy} and the inequality $T(\Om)\ge N \tan \frac{\pi}{N}$ (see Theorem \ref{th:T_estimates} and \eqref{eq:def_T} for the definition of $T(\Om)$). Indeed, for any $\Om\in\p_N$, one has
$$h(\Om) \leq  \frac{P(\Om)+\sqrt{P(\Om)^2-4\big(T(\Om)-\pi\big)|\Om|}}{2|\Om|}\leq \frac{P(\Om)+\sqrt{P(\Om)^2+4\big(\pi-N\tan{\frac{\pi}{N}}\big) |\Om|}}{2|\Om|}.$$
The first inequality is an equality if and only if $\Om$ is Cheeger-regular and the second one is an equality if and only if $T(\Om)= N \tan \frac{\pi}{N}$, which is equivalent to $\alpha_1=\dots=\alpha_N= \frac{N-2}{N}\pi$. 

\subsection{Proof of Theorem \ref{th:diagram_polygon}}

\subsubsection{If $N=3$}
We have by \eqref{eq:triangles_cheeger}, 
$$\forall \Om\in \p_3,\ \ \ \ \ \ \   \sqrt{|\Om|}h(\Om)= \frac{P(\Om)}{2\sqrt{|\Om|}}+\sqrt{\pi}.$$
Thus, we have the inclusion
$$\D_3\subset\left\{\left(x,\frac{x}{2}+\sqrt{\pi}\right)\Big|\ \ x\ge P(R_3)\right\}.$$
The { opposite} inclusion is proved by considering for example the family $(T_d)_{d\ge 1}$ of isosceles triangles of vertices $X_d\left(0,\frac{\sqrt{3}}{2}\right)$, $Y_d\left(\frac{d}{2},0\right)$ and $Z_d\left(-\frac{d}{2},0\right)$. We have for every $d\ge 1$,

 $$\left\{\begin{matrix}
 P(R_3)=x_1\leq x_d := \frac{P(T_d)}{\sqrt{|T_d|}} = \frac{d+\sqrt{d^2+3}}{\frac{3^{1/4}}{2}\sqrt{d}}\underset{d\rightarrow +\infty}{\longrightarrow} +\infty
\\ 
\\
    h(R_3)=y_1\leq y_d := \frac{P(T_d)}{2\sqrt{|T_d|}}+\sqrt{\pi}=\frac{d+\sqrt{d^2+3}}{3^{1/4}\sqrt{d}}+\sqrt{\pi}\underset{d\rightarrow +\infty}{\longrightarrow} +\infty,
\\
\end{matrix}\right.$$
where the inequalities $x_1\leq x_d$ and $y_1\leq y_d$ are consequences of the isoperimetric inequality for triangles.

\subsubsection{If $N$ is even}\label{ss:even}
We have by the inequalities \eqref{eq:new_inequality} and \eqref{eq:uper_bound}
$$\mathcal{D}_N \subset \left\{(x,y)\ |\ x\ge P(R_N)\ \ \text{and}\ \ \frac{x}{2}+\sqrt{\pi}\leq y\leq f_N(x)\right\},$$
where $f_N:x\in [P(R_N),+\infty)\longmapsto \frac{x+\sqrt{x^2+4(\pi-N\tan{\frac{\pi}{N}})}}{2}$. 

It remains to prove the { opposite} inclusion. We provide explicit families of elements of $\p_N$ that respectively fill the upper and lower boundaries of $\D_N$ (see Definition \ref{def:boundary_polygons}) and then use those domains to construct continuous paths that fill the diagram.

\vspace{2mm}
\underline{\textbf{Step 1: The upper boundary of $\D_{N}$:}}\vspace{2mm}

We recall that the inequality \eqref{eq:uper_bound} is an equality if and only if $\Om$ is Cheeger-regular and all its angles are equal to $(N-2)\pi/N$. { We assume without loss of generality that two parallel sides of the regular $N$-gon $R_N$ are colinear to the $x$-axis (note that this is possible because the number of sides $N$ is even). We then consider the family of $N$-gons $(\widetilde{U}_t)_{t\ge 1}$ such that for every  $t\ge 1$,
$$\widetilde{U}_t := \left\{\left(t x,y\right)\ |\ (x,y)\in R_N\right\}.$$
Since the map $t\ge 1\longmapsto d(\widetilde{U}_t)\ge d(R_N)$ is continuous and strictly increasing, it is a bijection. Thus, for every $\delta \ge d(R_N)$, there exists a unique $t_\delta\ge 1$ such that $d(\widetilde{U}_{t_\delta})= \delta$. From now on we denote $U_\delta := \widetilde{U}_{t_\delta}$ for every $\delta \ge d(R_N)$. 

Since 
\begin{itemize}
    \item $(U_\delta)_\del$ is a family of $N$-gons that vary continuously with respect to the Hausdorff distance,
    \item the perimeter and the area are continuous with respect to the Hausdorff distance,
    \item $P(U_{d(R_N)})= P(R_N)$ and 
    \begin{equation}\label{eq:estimate}
\frac{P(U_\del)}{|U_\del|^{1/2}}\ge \frac{P(U_{\del})}{\del^{1/2}  d(R_N)^{1/2}} \ge \frac{2\del}{\del^{1/2}  d(R_N)^{1/2}}= \frac{2}{d(R_N)^{1/2}} \delta^{1/2}\underset{\del \rightarrow +\infty}{\longrightarrow} +\infty,
\end{equation}
\end{itemize}
we have by the  Intermediate Value Theorem:}  $$\forall\ p \ge P(R_N),\exists\ \del_p\ge d(R_N), \ \ \ \ \ \ \frac{P(U_{\del_p})}{|U_{\del_p}|^{1/2}} = p.$$

Moreover, the sets $(U_\delta)$ are Cheeger-regular and all their interior angles are equal to $(N-2)\pi/N$. Thus, they all realize the equality
$$|U_\del|^{1/2}h(U_\del) = \frac{P(U_\del)+\sqrt{P(U_\del)^2+4\big(\pi-N\tan{\frac{\pi}{N}}\big) |U_\del|}}{2|U_\del|^{1/2}} = f_N\left(\frac{P(U_\delta)}{\sqrt{|U_\delta|}}\right).$$
We then deduce that the upper boundary of $\D_N$ is given by the set of points $\big\{\big(x,f_N(x)\big)\ \ |\ \ x\ge P(R_N)\big\}$.

\vspace{2mm}
\underline{\textbf{Step 2: The lower boundary of $\D_N$:}}\vspace{2mm}



As for the upper boundary's case, we construct a continuous family of $N$-gons $(V_\del)_{\del \ge d(R_N)}$, such that $V_{d(R_N)}=R_N$ and  $d(V_\del)=\del$ for every $\del\ge d(R_N)$. We assume that a diameter of $R_N$ is given by $[OA]$, where $O=(0,0)$ and $A=(d(R_N),0)$ and denote by $B_N$ its incircle (see Figure \ref{fig:lower_domains}) and $M_1,\dots,M_N$ its vertices.

Take $\del\ge d(R_N)$, we denote by $A_\del = (\del,0)$ and $(\Delta_\del),(\Delta_\del')$ the lines passing through $A_\del$ which are tangent to $B_N$. The line $(\Delta_\del)$ (resp. $(\Delta'_\del)$) cuts the boundary of $R_N$ in at least two points: we denote by $M_{k_\del}^\del$ (resp. $M_{N-k_\del+1}^\del$) the farthest one from $A_\del$ (see Figure \ref{fig:lower_domains}), where $k_\del\in \llbracket 1,N/2\rrbracket$ such that { $2k_\del$} is the number of vertices of $R_N$ that are in the region given by the convex cone delimited by $(\Delta_\del)$ and $(\Delta'_\del)$. We then define $V_\del$ as the (convex) polygon whose vertices are given by
{
$$
\left\{\begin{matrix}
M_1^\delta=M_1=O,\vspace{2mm}
\\ 
M_i^\del = M_i,\ \ \ \text{for all $k\in \llbracket 2,k_\del-1 \rrbracket$}\vspace{2mm}
\\ 
M_{k_\del}^\del = \dots = M^\delta_{\frac{N}{2}}\vspace{2mm}
\\
M_{\frac{N}{2}+1}^\del = A_\del\vspace{2mm}
\\
M_{\frac{N}{2}+2}^\del = \dots = M^\delta_{N+2-k_\del}\vspace{2mm}
\\
M_i^\delta = M_i \ \ \ \text{for all $k\in \llbracket N+2-k_\del,N-1 \rrbracket$}
\end{matrix}\right.
$$
}

\begin{figure}[h]
    \centering
    \includegraphics[scale=0.6]{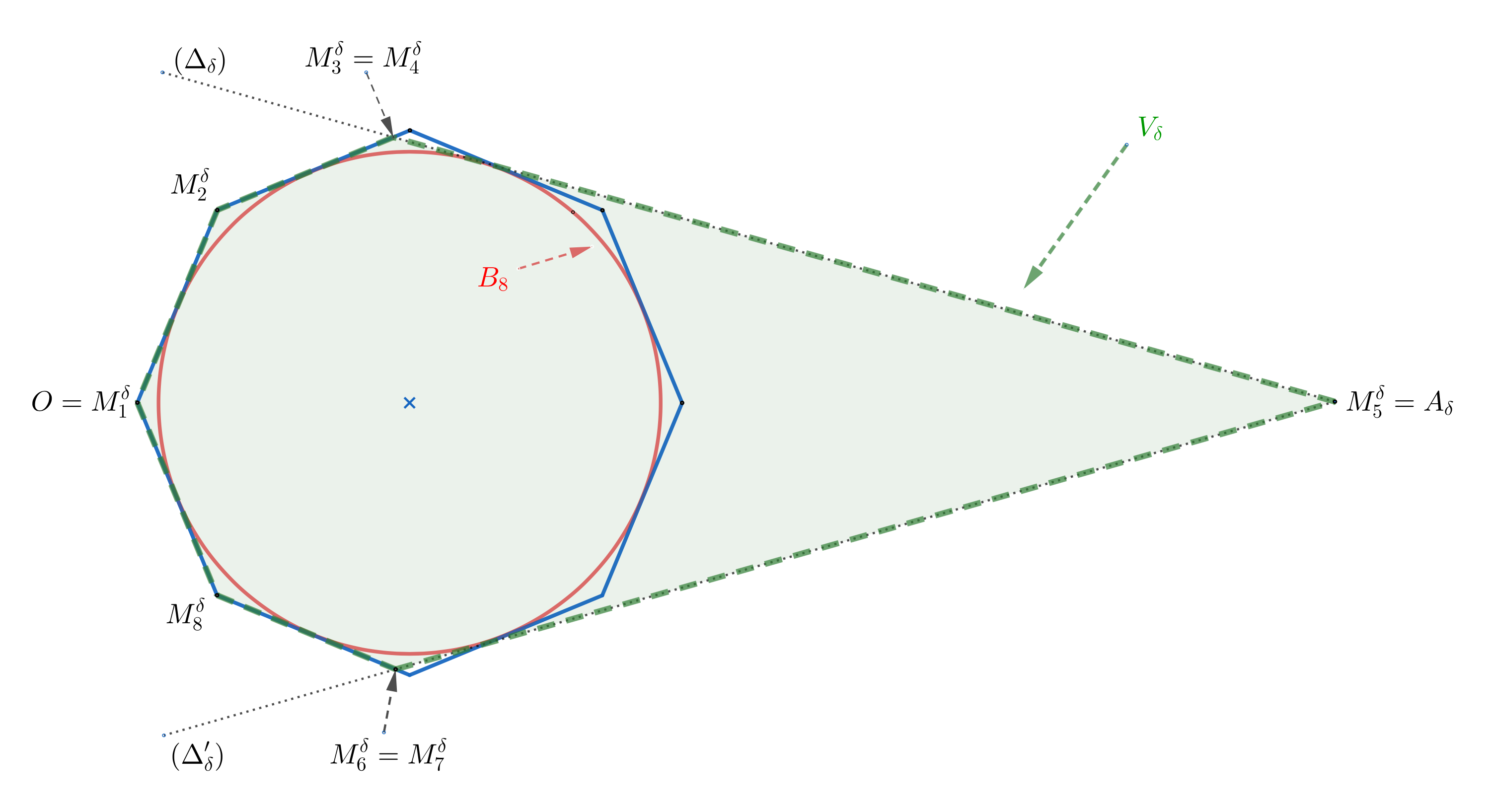}
    \caption{Construction of the circumscribed polygons $V_\del$ for $N=8$.}
    \label{fig:lower_domains}
\end{figure}

 Note that $V_\del$ has at most $N$ sides and that it is a circumscribed polygon. This yields that {the couple $(\frac{P(V_\del)}{|V_\del|^{1/2}},|V_\del|^{1/2}h(V_\del))$ lies} on the lower boundary of the diagram $\D_N$. We also, note that the applications $\del\in [d(R_N),+\infty)\longmapsto M_k^\del\in \R^2$ are continuous and thus { the family of polygons $(V_\del)_\del$ is continuous with respect to the Hausdorff distance}. Then, by { estimates similar to} \eqref{eq:estimate}, we get that $\lim\limits_{\del \rightarrow +\infty}\frac{P(V_\del)}{|V_\del|^{1/2}}= +\infty$. Thus, the lower boundary of $\D_N$ is given by the set of points $\big\{\big(x,x/2+\sqrt{\pi}\big)\ \ |\ \ x\ge P(R_N)\big\}$.
 
\vspace{2mm}

\underline{\textbf{Step 3: Continuous paths:}}\vspace{2mm}

Now that we have two families $(U_\del)$ and $(V_\del)$ of extremal shapes, it remains to define continuous paths that connect the upper domains to the lower ones and fill the whole diagram. Unfortunately, unlike for the case of the class $\K^2$, one cannot use Minkowski sums as they increase the number of sides and thus could give polygons that are not in the class $\p_N$, we will then construct the paths by continuously mapping the lower and upper polygons vertices. 

We assume without loss of generality that as for $V_\delta$ the diameter of $U_\del$ is given by $OA_\del$. We denote by  $O = L_1^\del, L_2^\del,\dots,L_{N/2-1}^\del=A_\del, L_{N/2}^\del,\dots,L_N^\del $ the vertices of  $U_\del$. For $t \in [0,1]$, we define $\Om_t^\delta$ as the polygon of vertices $((1-t)M_k^\delta+t L_k^\delta)_{ k\in [\![1,N]\!]}$. The polygon $\Om_t^\delta$ is convex and included in the rectangle $(0,\delta)\times\big(-\frac{d(R_N)}{2},\frac{d(R_N)}{2}\big)$. Thus, we have the following inequality :
\begin{equation}\label{eq:ok}
\forall t\in[0,1],\ \ \ \ \ \frac{P(\Om_t^\del)}{|\Om_t^\del|^{1/2}}\ge \frac{2\delta}{\delta^{1/2} d(R_N)^{1/2}}=\frac{2}{d(R_N)^{1/2}} \del^{1/2}.
\end{equation}

For every $\del\ge d(R_N)$, we introduce the closed and continuous path $\gamma_{\del}:[0,3)\longrightarrow \R^2$, defined as follows:  
$$\begin{array}{ccccc}
 t & \longmapsto & \begin{cases}
    \left(\frac{P(\Om^\del_{t})}{|\Om^\del_{t}|^{1/2}},|\Om^\del_{t}|^{1/2}h(\Om^\del_{t})\right), &\qquad\mbox{if $t\in [0,1],$}\vspace{1mm} \\
    \left((t-1)P(R_N)+(2-t)\frac{P(U_\del)}{|U_\del|^{1/2}}\ ,\ f_N\left((t-1)P(R_N)+(2-t)\frac{P(U_\del)}{|U_\del|^{1/2}}\right)\right),& \qquad\mbox{if $t\in (1,2],$}
 \\
    \left((3-t)P(R_N)+(t-2)\frac{P(V_\del)}{|V_\del|^{1/2}}\ ,\ (3-t)(\frac{P(R_N)}{2}+\sqrt{\pi})+(t-2)(\frac{P(V_\del)}{2|V_\del|^{1/2}}+\sqrt{\pi})\right),& \qquad\mbox{if $t\in (2,3).$} \vspace{1mm}
\end{cases}\\
\end{array}$$


\vspace{2mm}
\underline{\textbf{Step 4: Stability of the paths:}}\vspace{2mm}

Take $\del_0 \ge d(R_N)$ and $\eps > 0$, let us show that 
\begin{equation}\label{eq:convergencep}
\exists\ \alpha_\eps>0, \forall \del\in (\del_0-\alpha_\eps,\del_0+\alpha_\eps)\cap [P(R_N),+\infty), \ \ \ \ \ \ \ \ \ \sup_{t\in [0,3]}\|\ \gamma_\del(t)-\gamma_{\del_0}(t)\ \|\leq \eps.
\end{equation}
Let us take $\del\in [d(R_N),\del_0+1]$, with straightforward computations, there exists a constant $C(\del_0)>0$ depending only on $\del_0$ such that for every $t\in[1,3]$,
$$\|\gamma_\del(t)-\gamma_{\del_0}(t)\|\leq C(\del_0)\min\left(\left|\frac{P(U_\del)}{|U_\del|^{1/2}}-\frac{P(U_{\del_0})}{|U_{\del_0}|^{1/2}}\right|,\left|\frac{P(V_\del)}{|V_\del|^{1/2}}-\frac{P(V_{\del_0})}{|V_{\del_0}|^{1/2}}\right|\right)\underset{\del\rightarrow \del_0}{\longrightarrow} 0.$$
Moreover, by the quantitative estimates given in { Lemma \ref{lem:quantitative}}, there exist constants $C'({\delta_0}),C''(\delta_0)>0$, depending only on $\delta_0$, such that for all $\del\in [d(R_N),\delta_0+1]$ and all  $t\in [0,1]$ 
\begin{align*}
\|\gamma_\del(t)-\gamma_{\del_0}(t)\|&\leq \left|\frac{P(\Om^\del_t)}{|\Om^\del_t|^{1/2}}-\frac{P(\Om^{\del_0}_t)}{|\Om^{\del_0}_t|^{1/2}}\right|+\big| |\Om^\del_t|^{1/2}h(\Om^\del_t)-|\Om^{\del_0}_t|^{1/2}h(\Om^{\del_0}_t)\big|\\
& \leq C'(\delta_0) \left( |P(\Om_t^\del)-P(\Om_t^{\del_0})|+\left||\Om_t^\del|-|\Om_t^{\del_0}|\right|+|h(\Om_t^\del)-h(\Om_t^{\del_0})| \right) \\
&\leq C''(\delta_0) \max_{i \in \llbracket 1,N  \rrbracket} \|(1-t)M_i^\del+t L_i^\del-(1-t)M_i^{\del_0}-t L_i^{\del_0}\|\\
&\leq C''(\delta_0) \max_{i \in \llbracket 1,N  \rrbracket} (\|M_i^\del-M_i^{\del_0}\|+\|L_i^\del-L_i^{\del_0}\|) \underset{\del\rightarrow \del_0}{\longrightarrow} 0.
\end{align*}
Finally, we deduce that $\lim\limits_{\del\rightarrow \del_0} \sup\limits_{t\in [0,3]}\|\ \gamma_\del(t)-\gamma_{\del_0}(t)\|=0$, which proves \eqref{eq:convergencep}.

\vspace{2mm}
\underline{\textbf{Step 5: Conclusion:}}\vspace{2mm}

{ As for the case of convex sets (see Section \ref{sect:preuve_convex}), now that we proved that the boundaries $\{\big(x,f_N(x)\big)\ \ |\ \ x\ge P(R_N)\}$ and $\{(x,x/2+\sqrt{\pi})\ \ |\ \ x\ge P(R_N)\}$ are included in the diagram $\D_N$,  it remains to show that it is also the case for the set of points contained between them. Let  $A(x_A,y_A)\in \left\{(x,y)\ \ |\ \ x> P(R_N)\ \ \ \text{and}\ \ \  x/2+\sqrt{\pi}< y < f_N(x) \right\}$. Step 4 shows that for any choice of $\del$ and $\del'$ in $[d(R_N),\del_0+1]$, the curves $\gamma_{\del}$ and $\gamma_{\del'}$ are homotopic, and the homotopy is $H:(\sigma,t)\in [0,1]\times [0,3)\longmapsto \gamma_{(1-\sigma)\del+\sigma \del'}$. In particular, let us chose $\del=d(R_N)$ and $\del'=x_A^2d(R_N)$. The index of $A$ with respect to $\gamma_{\del}=\{(P(R_N),P(R_N)/2+\sqrt{\pi})\}$ is equal to $0$. Meanwhile, by the inequality \eqref{eq:ok}, we have that $A$ is in the interior of the curve $\gamma_{\del'}$, which means that its index with respect to $\gamma_{\del'}$ is non-zero. Thus, it must follow that there exists $(\Bar{\sigma},\Bar{t})\in [0,1]\times [0,3)$ such that $A=H(\Bar{\sigma},\Bar{t})\in \D_{N}$.
}




Finally, we get the equality 
$$\D_{N}=\left\{(x,y)\ \ |\ \ x\ge P(R_N)\ \ \ \text{and}\ \ \  \frac{1}{2}x+\sqrt{\pi}\leq y \leq \frac{x+\sqrt{x^2+4(\pi-N\tan{\frac{\pi}{N}})}}{2}\right\}.$$

\subsubsection{If $N\ge 5$ is odd}

By the inequalities \eqref{eq:new_inequality} and \eqref{eq:uper_bound}, we have
$$\D_N\subset \left\{(x,y)\ |\ x\ge P(R_N)\ \text{and}\  x/2+\sqrt{\pi}\leq y \leq f_N(x)\right\},$$
where 
\begin{equation}\label{eq:f_N}
f_N:x\longmapsto \frac{x+\sqrt{x^2+4(\pi-N\tan{\frac{\pi}{N}})}}{2}.    
\end{equation}

Let us study the lower and upper boundaries of the diagram $\D_N$.\vspace{2mm}

\underline{\textbf{The lower boundary of the diagram $\D_N$:}}\vspace{2mm}

Since $N-1$ is even, we have by Section \ref{ss:even} 
$$\{(x,x/2+\sqrt{\pi})\ |\ x\ge P(R_{N-1})\}\subset \D_{N-1}\subset \D_{N}.$$
It remains to prove that $$\{(x,x/2+\sqrt{\pi})\ |\ x\in[ P(R(N),P(R_{N-1})]\}\subset  \D_{N}.$$
To do so, we continuously move two consecutive sides of the polygon $R_N$ so as to align them while keeping the polygon circumscribed. This gives us a continuous ({ with respect to} the Hausdorff distance) family $(W_t)_{t\in[0,1]}$ of convex { circumscribed}  polygons such that $W_0:=R_N$ and $W_1$ is an element of $\p_{N-1}$,  see Figure \ref{fig:even}. 

\begin{figure}[h]
    \centering
    \includegraphics[scale=0.5]{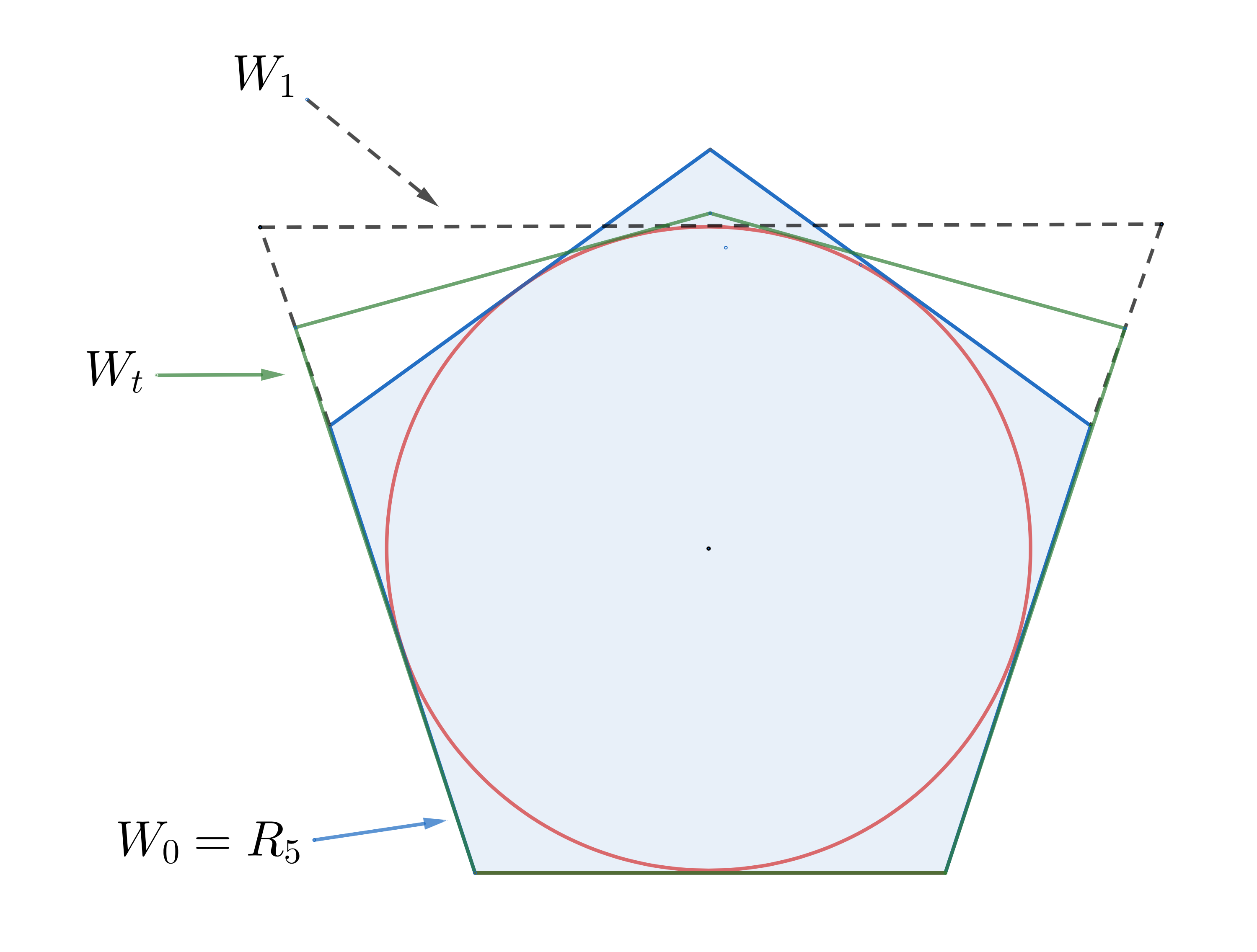}
    \caption{Construction of the circumscribed polygons $W_t$.}
    \label{fig:even}
\end{figure}

Since the { family of convex polygons $(W_t)_{t\in[0,1]}$ and the functionals perimeter, area and Cheeger constant are continuous  with respect to the Hausdorff distance} and $\frac{P(W_1)}{\sqrt{|W_1|}}\ge P(R_{N-1})$ (because of the polygonal isoperimetric inequality in $\p_{N-1}$), we have by the { Intermediate Value Theorem}
$$\{(x,x/2+\sqrt{\pi})\ |\ x\in [P(R_N),P(R_{N-1})]\}\subset \left\{\left(\frac{P(W_t)}{\sqrt{|W_t|}},\sqrt{|W_t|}h(W_t)\right)\ \Big|\ t\in [0,1]\right\}\subset \D_N.$$

We finally have 
$$\{(x,x/2+\sqrt{\pi})\ |\ x\in[ P(R(N),+\infty)\}\subset  \D_{N}.$$\vspace{2mm}

\newpage

\underline{\textbf{The upper boundary of $\D_N$:}}\vspace{2mm}

Let us now study the upper boundary of $\D_N$. We recall that the function $g_N$ is defined as follows:
\begin{equation}\label{eq:f}
\begin{array}{cccl}
   g_N : &[P(R_N),+\infty)& \longrightarrow&\mathbb{R}\\
   & p &\longmapsto &\sup\left\{h(\Omega)\ |\  \Omega\in \p_N,\ \  |\Om|=1\ \text{and}\ P(\Omega)= p\right\}.
\end{array}
\end{equation}
First, let us prove that the problem $\sup\left\{h(\Omega)\ |\ \Omega\in \p_N,\ \  |\Om|=1\ \text{and}\ P(\Omega)= p\right\}$ admits a solution, that we denote { by} $\Om_p\in \p_N$.  \vspace{2mm}

Take $(\Om_p^n)_{n\in\N}$ a sequence of elements of $\p_N$ such that $|\Om_p^n|=1$ and $P(\Om_p^n)=p$ for every $n\in\N$, which satisfies $$\lim\limits_{n\rightarrow +\infty} h(\Om_p^n)=\sup\left\{h(\Omega)\ |\ \Omega\in \p_N,\ \  |\Om|=1\ \text{and}\ P(\Omega)= p\right\}.$$

Since the diameters of the sets $(\Om_p^n)$ are all bounded by $p$ and the involved functionals are invariant by translations, we may assume without loss of generality that there exist a fixed ball $D\subset \R^2$ that contains all the polygons $\Om_p^n$. 

{ Let $n\in \N$. Since $\Om_p^n\in \p_N$, the polygon $\Om_p^n$ is the convex hull of $N$ points $A_1^n$, $A_2^n$,$\dots$, $A_{N}^n$}. The sequences $(A_1^n),\dots,(A_N^n)$ are bounded in $\R^2$. Thus, by Bolzano-Weirstrass Theorem, there exist $\sigma:\N\longrightarrow \N$ strictly increasing and $A_1,\dots,A_N\in \R^2$ such that $\lim\limits_{n\rightarrow +\infty}A_k^{\sigma(n)} = A_k$. By elementary arguments of convex geometry one shows that the { convex hull of the points $A_1,\dots,A_N$ defines a} convex polygon $\Om_p$ which is also the limit  of $(\Om_p^{\sigma(n)})_n$ { with respect to} the Hausdorff distance { (we refer to \cite[Section 1.8]{schneider} for results on the Hausdorff metric)}. By the continuity of the perimeter, the area and the Cheeger constant { with respect to} the Hausdorff distance among convex sets (see \cite[Proposition 3.1]{reverse_cheeger} for the continuity of the Cheeger constant), we have 
$$
\left \{
\begin{array}{l}
    |\Omega_p| = \lim\limits_{n\rightarrow+\infty}  |\Omega_{p}^{\sigma(n)}| =1,\\[3mm]
    P(\Omega_p)=  \lim\limits_{n\rightarrow+\infty}  P\big(\Omega_{p}^{\sigma(n)}\big) =p,\\[3mm]
   h(\Omega_p)=\lim\limits_{n\rightarrow +\infty} h\big(\Omega_{p}^{\sigma(n)}\big)=\sup\left\{h(\Omega)\ |\ \Omega\in \p_N,\ \  |\Om|=1\ \text{and}\ P(\Omega)= p\right\}.
\end{array}
\right.
    $$  
    
Finally, we conclude that $\Om_p\in \p_N$ is a solution of the problem $\sup\left\{h(\Omega)\ |\ \Omega\in \p_N,\ \  |\Om|=1\ \text{and}\ P(\Omega)= p\right\}$.\vspace{2mm}

Now, let us prove the properties of the function $g_N$ stated in Theorem \ref{th:diagram_polygon}.

\vspace{2mm}
\underline{\textbf{1) The function $g_N$ is continuous}}\vspace{2mm}

Let $p_0\in [P(R_N),+\infty)$.

\begin{itemize}
\item We first show { the} \textbf{superior limit inequality}. Let $(p_n)_{n\ge 1}$ { be a} real sequence converging to $p_0$ such that $$\limsup\limits_{p\rightarrow p_0} h(\Omega_p)=\lim\limits_{n\rightarrow +\infty} h(\Omega_{p_n}).$$
As the perimeters of $(\Om_{p_{n}})_{n\in\N^*}$ are uniformly bounded, one may assume that the domains $(\Omega_{p_n})_{n\in \mathbb{N}^*}$ are included in a fixed ball. Then by similar arguments as above, the sequence $(\Omega_{p_n})$ converges { with respect to} the Hausdorff distance up to a subsequence (that we also denote by $p_{n}$ for sake of simplicity) to a convex polygon $\Om^*\in\p_N$. 

Again, by the continuity of the perimeter, the area and the Cheeger constant { with respect to} the Hausdorff distance among convex sets (see \cite[Proposition 3.1]{reverse_cheeger} for the continuity of the Cheeger constant), we have: 
$$
\left \{
\begin{array}{l}
    |\Omega^*| = \lim\limits_{n\rightarrow+\infty}  |\Omega_{p_n}| =1,\\[3mm]
    P(\Omega^*)=  \lim\limits_{n\rightarrow+\infty}  P(\Omega_{p_n}) = \lim\limits_{n\rightarrow+\infty}  p_n=p_0,\\[3mm]
   h(\Omega^*)=\lim\limits_{n\rightarrow +\infty} h(\Omega_{p_n})=\limsup\limits_{p\rightarrow p_0} h(\Omega_p).
\end{array}
\right.
    $$  
    
 Then by the definition of $g_N$ (since $\Omega^*\in \p_N$, $|\Om^*|=1$ and $P(\Omega^*)=p_0$), we obtain

$$g_N(p_0)\geq\ h(\Omega^*)=\lim\limits_{n\rightarrow +\infty} h(\Omega_{p_n})=\limsup\limits_{p\rightarrow p_0} h(\Omega_p)=\limsup\limits_{p\rightarrow p_0} g_N(p).$$

\item It remains to prove \textbf{{ the} inferior limit inequality}. Let $(p_n)_{n\ge1}$ be a real sequence converging to $p_0$ such that
$$\liminf\limits_{p\rightarrow p_0} g_N(p) = \lim\limits_{n\rightarrow +\infty} g_N(p_n).$$
By using parallel chord movements (see the proof of Lemma \ref{lem:dorin}), we can construct a sequence of unit area polygons $(K_n)_{n\ge1}$ with at most $N$ sides, converging to $\Omega_{p_0}$ { with respect to} the Hausdorff distance such that $P(K_n)=p_n$ for sufficiently high values of $n\in \mathbb{N}^*$. By using the definition of $g_N$, one has $$\forall n\in \mathbb{N}^*,\ \ \ \ g_N(p_n)\geq h(K_n).$$

Passing to the limit, we get 
$$\liminf\limits_{p\rightarrow p_0} g_N(p) = \lim\limits_{n\rightarrow +\infty} g_N(p_n)\geq \lim\limits_{n\rightarrow +\infty} h(K_n)=h(\Omega_{p_0})=g_N(p_0).$$
\end{itemize}

We finally get that $\lim\limits_{p\rightarrow p_0} g_N(p) = g_N(p_0)$, so $g_N$ is continuous on $[P(R_N),+\infty)$.

\vspace{2mm}

\underline{\textbf{2) The function $g_N$ is strictly increasing}}\vspace{2mm}

Let us assume by contradiction that $g_N$ is not strictly increasing. Then, there exist $p_2>p_1\geq  P(R_N) $ such that { $g_N(p_2)\leq g_N(p_1)$}, and from the equality case in the polygonal isoperimetric inequality, we necessarily have $p_{1}>P(R_N)$. Since $g$ is continuous, it reaches its maximum on $[P(R_N),p_2]$ at a point $p^*\in(P(R_N),p_{2})$, that is to say
\begin{equation}\label{eq:locmax}
\forall \Om\in\p_N\textrm{ such that }|\Om|=1\ \ \textrm{and}\ \ P(\Om)\in[P(R_N),p_2], \;\;h(\Om_{p^*})=g_N(p^*)\geq h(\Om).
\end{equation}

We note that $g_N(p^*)>p^*/2+\sqrt{\pi}$. Indeed, if it is not the case (i.e., $g_N(p^*)=p^*/2+\sqrt{\pi}$), we have for { sufficiently small $t>0$   $$g_N(p^*+t)\ge (p^*+t)/2+\sqrt{\pi}>p^*/2+\sqrt{\pi}=g_N(p^*),$$}
which contradicts the fact that $g_N$ admits a local maximum at $p^*$.\vspace{2mm}  

{ The assertion} \eqref{eq:locmax} shows that $\Om_{p^*}$ is a local maximizer of the Cheeger constant between convex $N$-gons of unit area. On the other hand, the fact that { $h(\Om_{p^*})=g_N(p^*)>\frac{P(\Om_{p^*})+\sqrt{4\pi }}{2}$} implies that $\Om_{p^*}$ is not a circumscribed polygon (otherwise, it should satisfy the equality \eqref{eq:egal}). Let us now show that any non-circumscribed polygon $\Om$ (i.e., $T(\Om)<\frac{P(\Om)^2}{4|\Om|}$) can be locally perturbed (while preserving the number of sides) in such a way to increase $|\Om|^{1/2}h(\Om)$.\vspace{2mm}

We denote { by} $(\ell_i)_{i\in [\![1,N]\!]}$ the lengths of the sides of the polygon $\Om$ and $(\alpha_i)_{i\in [\![1,N]\!]}$ its interior angles and take $$r_0:=\min_{1\leq i\leq N} \frac{\ell_i}{\tan{\left(\frac{\pi}{2}-\frac{\alpha_{i}}{2}\right)}+\tan{\left(\frac{\pi}{2}-\frac{\alpha_{i-1}}{2}\right)}},$$
{ where we define $\alpha_0:=\alpha_N$. For every $i\in \llbracket 1,N \rrbracket$ and $\eps\in \R$ such that $|\eps|$ is sufficiently small, we introduce the polygon $\Om_\eps^i$ obtained by performing a parallel displacement of the $i$-th side with the algebraic distance $\eps$, see Figure \ref{fig:paralel}.}

\begin{center}
\begin{tikzpicture}
  \draw
   (1,2) coordinate (a) 
    -- (2,4) coordinate (b)
    -- (6,4) coordinate (c)
    pic[ draw=black,  angle eccentricity=1.2,line width=.2mm, angle radius=.3cm]
    {angle=a--b--c};
    
  \draw
   (1,2) coordinate (a) 
    -- (2,4) coordinate (b)
    -- (6,4) coordinate (c)
    pic[ draw=black,  angle eccentricity=1.2,line width=.2mm, angle radius=.3cm]
    {angle=a--b--c};

  \draw
   (4,4)  coordinate (a) 
    -- (6,4) coordinate (b)
    -- (7.5,2) coordinate (c)
    pic[ draw=black,  angle eccentricity=1.2,line width=.2mm, angle radius=.3cm]
    {angle=a--b--c};
    
  \draw
   (4,4)  coordinate (a) 
    -- (6,4) coordinate (b)
    -- (7.5,2) coordinate (c)
    pic[ draw=black,  angle eccentricity=1.2,line width=.2mm, angle radius=.25cm]
    {angle=a--b--c};    
    
\draw (1,2) -- (2,4);
\draw (6,4) -- (2,4);
\draw [dashed,blue,thick] (1.5,3) -- (6.75,3);
\draw [dashed,blue,thick] (1.5,3) -- (1,2);
\draw [dashed,blue,thick] (7.5,2) -- (6.75,3);
\draw [dashed,blue,thick] (1.5,3) -- (6.75,3);
\draw [dashed,blue,thick] (1.5,3) -- (1,2);
\draw [dashed,blue,thick] (7.5,2) -- (6.75,3);
\draw (6,4) -- (7.5,2);
\draw [dashed] (2,4) -- (2.5,5);
\draw [dashed] (6,4) -- (5.25,5);
\draw [dashed] (2.5,5) -- (5.25,5);
\draw [dotted,->](1,4) -- (1,5);
\draw [dotted,->](.5,4) -- (.5,3);
\draw [dashed,->](7,5.25) -- (5.85,4.75);
\draw [dashed,->,blue](8.5,3.25) -- (7.35,2.75);
\draw [<-](6.75,3.5) -- (8,4);
\draw (7,5.34) node[right] {$\Om_\eps^ i$ (if $\eps>0$)};
\draw (8.8,3.45)[blue] node[right] {$\Om_\eps^i$ (if $\eps<0$)};
\draw (1,4.5) node[left] {$\eps>0$};
\draw (.5,3.5) node[left] {$\eps<0$};
\draw (8,4.2) node[right] {$\Om$};
\draw (2.4,3.5) node {$\alpha_{i-1}$};
\draw (5.8,3.5) node {$\alpha_{i}$};
\end{tikzpicture}

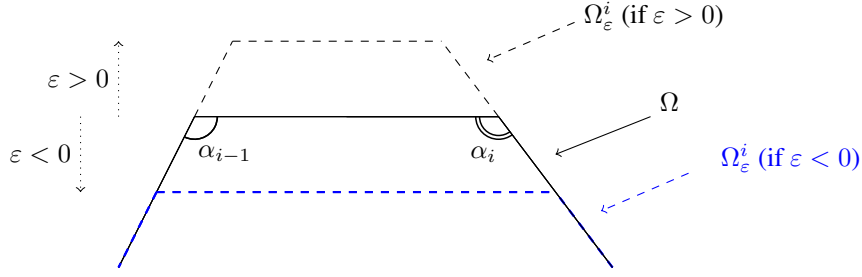
\captionof{figure}{Parallel displacement of the $i$-th side.}
\label{fig:paralel}
\end{center}

Let us distinguish two cases:
\begin{itemize}
\item if $|\Om|-r_0P(\Om)+r_0^2(T(\Om)-\pi)\ge 0$  (where $T$ is defined in \eqref{eq:def_T}), this means by \cite[Theorem 3]{kawohl} that there exists $i\in \llbracket 1,N \rrbracket$ such that the $i$-th side of $\Om$ that does not touch the boundary of the Cheeger set $C_\Om$ or touch it in one point. For $\eps>0$ sufficiently small, the polygons $\Om$ and $\Om_\eps^ i$ have the same Cheeger set. Thus, we have 
$|\Om_\eps^i|^{1/2}h(\Om_\eps^i) > |\Om|^{1/2}h(\Om)$.

\item On the other hand if $|\Om|-r_0P(\Om)+r_0^2(T(\Om)-\pi)< 0$ (where $T$ is defined in \eqref{eq:def_T}), then by \cite[Theorem 3]{kawohl}, the polygons $(\Om_\eps^i)_{i\in \llbracket 1,N \rrbracket}$ (for $|\eps|$ sufficiently small) are Cheeger-regular and thus we have explicit formulas for their Cheeger constants. We write
\begin{equation}\label{eq:cheeger}
|\Om_\eps^i|^{1/2}h(\Om_\eps^i) = \frac{P(\Om_\eps^i)+\sqrt{P(\Om_\eps^i)^2-4(T(\Om_\eps^i)-\pi)|\Om_\eps^i|}}{\sqrt{|\Om_\eps^i|}}=\frac{P(\Om_\eps^i)}{\sqrt{|\Om_\eps^i|}}+\sqrt{\frac{P(\Om_\eps^i)^2}{|\Om_\eps^i|}-4(T(\Om)-\pi)},
\end{equation}
where we used $T(\Om)=T(\Om_\eps^i)$ for the last equality.

As stated in the proof of \cite[Lemma 23]{bucur_fragala}, through elementary geometric arguments, we have { for every $i\in \llbracket 1,N\rrbracket$,}

 $$
\left \{
\begin{array}{c @{=} c}
    P(\Om_\eps^i)\ \ \ &\ \ \ P(\Om) + \left(\frac{1}{\tan \alpha_{i-1}}+\frac{1}{\tan \alpha_i}+\frac{1}{\sin \alpha_{i-1}}+\frac{1}{\sin \alpha_i}\right)  \eps, \vspace{3mm}\\ 
    |\Om_\eps^i|\ \ \ &\ |\Om| + \ell_i \eps + \frac{1}{2}\left(\frac{1}{\tan \alpha_{i-1}}+\frac{1}{\tan \alpha_i}\right) \eps^2.
\end{array}
\right.
    $$   
Thus,
{
\begin{align*}
\frac{P(\Om_\eps)^2}{|\Om_\eps|}& =  \frac{\left(P(\Om) + \left(\frac{1}{\tan \alpha_{i-1}}+\frac{1}{\tan \alpha_i}+\frac{1}{\sin \alpha_{i-1}}+\frac{1}{\sin \alpha_i}\right)  \eps\right)^2}{|\Om| + \ell_i \eps + \frac{1}{2}\left(\frac{1}{\tan \alpha_{i-1}}+\frac{1}{\tan \alpha_i}\right) \eps^2}\\
& = \frac{P(\Om)^2}{|\Om|} +P(\Om)\cdot \Psi_i\cdot\eps + \underset{\eps\rightarrow 0}{o}(\eps),
\end{align*}
where $$\Psi_i:=2\left(\frac{1}{\tan \alpha_{i-1}}+\frac{1}{\tan \alpha_{i}}+\frac{1}{\sin \alpha_{i-1}}+\frac{1}{\sin \alpha_i}\right)-\frac{P(\Om)}{|\Om|}\ell_i.$$}

Let us show that there exists $i\in [\![1,N]\!]$ such that $\Psi_i \ne 0$. We assume by contradiction that $\Psi_i= 0$ for every $i\in [\![1,N]\!]$, we then have $$\sum_{i=1}^N \Psi_i= 0,$$
which is equivalent to $$P(\Om) = \frac{4|\Om|}{P(\Om)} \sum_{i=1}^N\left(\frac{1}{\tan \alpha_i}+\frac{1}{\sin \alpha_i}\right)=\frac{4|\Om|}{P(\Om)} \sum_{i=1}^N \frac{1}{\tan \frac{\alpha_i}{2}}=\frac{4|\Om|}{P(\Om)}\cdot T(\Om),$$
where $T(\Omega)$ is defined in \eqref{eq:def_T}. 
As stated in Theorem \ref{th:T_estimates}, this equality holds if and only if $\Om$ is a circumscribed polygon, which is not the case as assumed above. Thus, there exists $i\in [\![1,N]\!]$ such that $\Psi_i \ne 0$. Then, by performing a parallel displacement (in the suitable sense: $\eps>0$ if $\psi_i>0$ and $\eps<0$ if $\psi_i<0$) of the $i^{\text{th}}$ side, one is able to strictly increase $\frac{P(\Om)}{|\Om|^{1/2}}$ and thus, by \eqref{eq:cheeger}, increase $|\Om|^{1/2}h(\Om)$.
\end{itemize}

\vspace{2mm}
\underline{\textbf{3) Comparison between $g_N$ and $f_N$ and asymptotic}}
\begin{itemize}
\item It is immediate by the inclusion $\p_{N-1}\subset \p_N$, the equality $f_{N-1}=g_{N-1}$ (because $N-1$ is even) and the inequality \eqref{eq:uper_bound} that $f_{N-1}\leq g_N\leq f_N$ on $[P(R_N),+\infty)$ (we recall that $f_N$ is defined in \eqref{eq:f_N}). 
\item { If we perform a parallel displacement of one of the sides of the regular polygon $R_N$,  we see that there exists $\eps_0>0$ and a continuous family (with respect to the Hausdorff distance) of Cheeger-regular polygons $(\Om_\eps)_{\eps\in [0,\eps_0)}$} with the same interior angles as $R_N$ such that $\frac{P(\Om_\eps)}{|\Om_\eps|^{1/2}} > \frac{P(\Om)}{|\Om|^{1/2}}$, for every $\eps\in(0,\eps_0)$. This shows that there exists $b_N\ge \frac{P(\Om_{\eps_0})}{|\Om_{\eps_0}|^{1/2}}> P(R_N)$ such that 
$$\forall x\in [P(R_N),b_N],\ \ \ \  g_N(x) = \frac{x+\sqrt{x^2+4(\pi-N\tan{\frac{\pi}{N}})}}{2} =f_N(x).$$

\item 

Let us now prove that if $\Om$ is a polygon of $N$ sides and unit area and whose angles are all equal (to $\beta_N:=\frac{(N-2)\pi}{N}$), one has 
\begin{equation}\label{ineq:c_N}
P(\Om)\leq 2N\sqrt{\tan \frac{\beta_N}{2}}.
\end{equation}
We assume that the polygon $\Om$ is included in the half-plane $\{y\ge0\}$ and that its longest side is given by the segment $[OA]$, where $A(\ell,0)$ and $\ell>0$.  Since $N$ is odd and all the angles of $\Om$ are equal, we deduce that there exists a unique vertex $B(x_B,\eta)$ which is strictly higher (i.e., has the largest ordinate) than all the other vertices. We can assume without loss of generality that $x_B\ge \ell/2$. { We denote by $C(x_C,0)$ the point of intersection of the line obtained by extending the left side of extremity $B$ and the $x$-axis and by $D(x_B,0)$ the orthogonal projection of the point $B$ on the $x$-axis, see Figure \ref{fig:flat_odd_polygon}. By the convexity of $\Om$,  we have $0<\theta\leq \frac{\beta_N}{2}<\frac{\pi}{2}$, where $\theta$ is the angle between the vectors $\overrightarrow{BO}$ and $\overrightarrow{BD}$. 

We note that $x_C\leq 0$. Indeed, 
\begin{equation}\label{eq:x_C}
x_C = x_B - (x_B-x_C) = x_B - CD = x_B - \eta \tan{\frac{\beta_N}{2}} \leq x_B - \eta \tan{\theta} = x_B - OD = 0.    
\end{equation}
}

\begin{center}
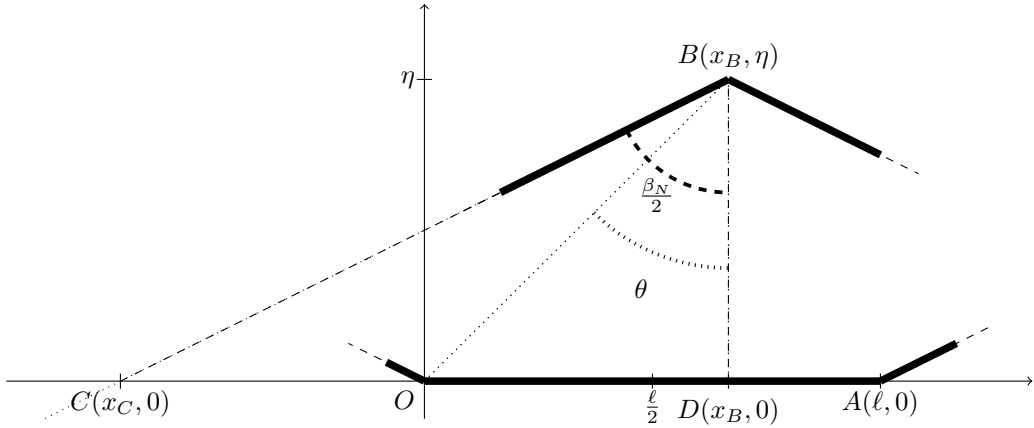

\begin{tikzpicture}
  \draw[dotted]
   (0,0) coordinate (a) 
    -- (4,4) coordinate (b)
    -- (4,0) coordinate (c)
    pic["$\theta$", draw=black,  angle eccentricity=1.2,line width=.5mm, angle radius=2.5cm]
    {angle=a--b--c};
    
      \draw[dashed]
   (-4,0) coordinate (a) 
    -- (4,4) coordinate (b)
    -- (4,0) coordinate (c)
    pic["$\frac{\beta_N}{2}$", draw=black, line width=.5mm, angle eccentricity=1.2, angle radius=1.5cm]
    {angle=a--b--c};

\draw[->] (-5.5,0) -- (8,0);
\draw (6,0) node[below] {$A(\ell,0)$};
\draw (0,0) node[below left] {$O$};
\draw (3,0) node[below] {$\frac{\ell}{2}$};
\draw (4,-.1) node[below] {$D(x_B,0)$};
\draw (-4,0) node[below] {$C(x_C,0)$};
\draw (4,4) node[above] {$B(x_B,\eta)$};
\draw (0,4) node[left] {$\eta$};
\draw [->] (0,-.5) -- (0,5);
\draw [dotted] (0,0) -- (4,4);
\draw [dotted] (4,4) -- (4,0);
\draw [dashed] (6,0) -- (7.5,.75);
\draw [dashed] (0,0) -- (-1,.5);
\draw [line width=1mm] (6,0) -- (7,.5);
\draw [line width=1mm] (0,0) -- (-.5,.25);
\draw [line width=1mm] (0,0) -- (6,0);
\draw [line width=1mm] (4,4) -- (6,3);
\draw [line width=1mm] (4,4) -- (1,2.5);
\draw [dashed] (4,4) -- (.5,2.25);
\draw [dotted] (4,4) -- (-5,-.5);
\draw [dashed] (4,4) -- (6.5,2.75);

\draw (6,-0.1) -- (6,0.1);
\draw (-0.1,4) -- (0.1,4);
\draw (4,-0.1) -- (4,0.1);
\draw (3,-0.1) -- (3,0.1);
\draw (-4,-0.1) -- (-4,0.1);

\end{tikzpicture}
\captionof{figure}{An $N$-gon with all interior angles equal to $\beta_N$.}
\label{fig:flat_odd_polygon}
\end{center}

We have $$\frac{1}{\tan \frac{\beta_N}{2}}=\cotan \frac{\beta_N}{2}=\frac{\eta}{x_B-x_C}=2\frac{\frac{1}{2}\ell \eta}{\ell(x_B-x_C)}= 2\frac{\mathcal{S}_{OAB}}{\ell(x_B-x_C)} \leq \frac{4}{\ell^2},$$
{ where $\mathcal{S}_{OAB}$ corresponds to the area of the triangle $OAB$}. The last inequality is a consequence of $\mathcal{S}_{OAB}\leq 1$ (because $OAB\subset \Om$ and $|\Om|=1$) and $x_B-x_C\ge x_B\ge \ell/2$ { (we recall that $x_C\leq 0$ as shown in \eqref{eq:x_C})}.

Thus, we have the result
$$P(\Om) \leq N\ell \leq 2N{\sqrt{\tan \frac{\beta_N}{2}}}.$$
This proves that there is no polygon of unit area, $N$ sides and perimeter larger than  $2N{\sqrt{\tan \frac{\beta_N}{2}}}$ whose interior angles are all equal (to $\beta_N$). Thus, for every $\Om\in\p_N$ such that $|\Om|=1$ and $P(\Om)> 2N{\sqrt{\tan \frac{\beta_N}{2}}}$, we have 
$$h(\Om) \leq  \frac{P(\Om)+\sqrt{P(\Om)^2-4\big(T(\Om)-\pi\big)|\Om|}}{2|\Om|}<\frac{P(\Om)+\sqrt{P(\Om)^2+4\big(\pi-N\tan{\frac{\pi}{N}}\big) |\Om|}}{2|\Om|},$$
where the first inequality corresponds to the inequality \eqref{eq:uper_bound}  and the second (strict) one is a consequence of the inequality $T(\Om)>N\tan \frac{\pi}{N}$ of Theorem \ref{th:T_estimates} (see \eqref{eq:def_T} for the definition of $T(\Omega)$). 

We finally have that $$\forall x> 2N{\sqrt{\tan \frac{\beta_N}{2}}},\ \ \ \ \ \ \ g_N(x)<\frac{x+\sqrt{x^2+4(\pi-N\tan{\frac{\pi}{N}})}}{2}.$$
\item Since $N\ge 4$, we have 
$$\forall x\ge P(R_{N-1}),\ \ \ \ \ \ \  \frac{x+\sqrt{x^2+4(\pi-(N-1)\tan{\frac{\pi}{N-1}})}}{2} \leq  g_N(x) \leq \frac{x+\sqrt{x^2+4(\pi-N\tan{\frac{\pi}{N}})}}{2}. $$
Thus, $$g_N(x)\underset{x\rightarrow +\infty}{\sim} x.$$
\end{itemize}

\section{Numerical simulations}\label{s:numeric}
Since it is not easy to give an explicit description of the upper boundary of the diagram $\D_N$ when $N$ is odd, we perform some simulations in order to have an approximation of the function $g_N$ (defined in Definition \ref{def:boundary_polygons}). We numerically solve the following problems: 
\begin{equation}\label{eq:problem}
\max\left\{h(\Omega)\ |\ \Omega\in \p_N,\ \  |\Om|=1\ \text{and}\ P(\Omega)= p_0\right\},
\end{equation}
where $p_0\in [P(R_N),+\infty)$. 
\subsection{Parametrization of the domains}
We parametrize a polygon $\Om$ via its vertices' coordinates $A_1:=(x_1,y_1),\dots,A_N:=(x_N,y_N)$.  
\begin{itemize}
\item Let us first express the constraint of convexity in terms of the coordinates of the vertices of $\Om$. It is classical that a polygon $\Om$ is convex if and only if all the interior angles are less than or equal to $\pi$. By using the cross product (see \cite{MR4089508} for example), the convexity is equivalent to the constraints
$$C_k(x_1,\dots,x_N,y_1,\dots,y_N) := (x_{k-1}-x_k)(y_{k+1}-y_k)-(y_{k-1}-y_k)(x_{k+1}-x_k)\leq 0,$$
for $k \in \llbracket 1,N \rrbracket$, where we used the conventions $A_0 = A_N$ and $A_{N+1} = A_1$. 
\item The area and the perimeter of $\Om$ are given by the following formulas 
 $$
\left \{
\begin{array}{c @{:=} c}
    f(x_1,\dots,x_N,y_1,\dots,y_N)\ \ \ &\ \ \ P(\Om)\  = \sum_{k=1}^N \sqrt{(x_{k+1}-x_k)^2+(y_{k+1}-y_k)^2}, \vspace{3mm}\\ 
    g(x_1,\dots,x_N,y_1,\dots,y_N)\ \ \ &\ |\Om|\ =\ \frac{1}{2}\left| \sum_{k=1}^N (x_k y_{k+1}-x_{k+1}y_k) \right|\ \ \ \ \ \ \ \ \ \ \ \ \ \ \ \ \ \ \ \
\end{array}
\right.
    $$   
\item Finally, we introduce the function 
$$\phi: (x_1,\dots,x_N,y_1,\dots,y_N)\longmapsto \begin{cases}
    h(\Omega), &\qquad\mbox{if the polygon $\Omega$ does not have overlapping sides,}\vspace{1mm} \\
    -1,& \qquad\mbox{if the polygon $\Omega$ { has} overlapping sides,}
\end{cases}$$
where $\Om$ is the polygon of vertices $A_1(x_1,y_1),\dots,A_N(x_N,y_N)$. The Cheeger constant is computed by using an open source Matlab code of B. Bogosel \cite{beni}. The algorithm combines the well known result of Kawohl and Lachand-Robert \cite{kawohl} (stated in Theorem \ref{th:KLR_convex}) which characterizes the Cheeger sets for convex domains and the toolbox Clipper, a very good implementation of polygons' inner parallel sets \footnote{We refer to Definition \ref{def:inner} for the notion of inner parallel sets.} computation by A. Johnson. 
\end{itemize}

We are now able to write the problem \eqref{eq:problem} in the following form 
$$
\left\{\begin{matrix}
\sup\limits_{(x_1,\dots,y_N)\in \R^{2N}} \phi(x_1,\dots,y_N),\vspace{2mm}
\\ 
\forall k\in \llbracket 1,N\rrbracket,\ \ \ \ C_k(x_1,\dots,x_N,y_1,\dots,y_N)\leq 0,\vspace{2mm}
\\ 
f(x_1,\dots,x_N,y_1,\dots,y_N)=p_0\vspace{2mm}
\\
g(x_1,\dots,x_N,y_1,\dots,y_N)=1.
\end{matrix}\right.
$$
\subsection{Computation of the gradients}

We want to use Matlab's routine \texttt{fmincon} to solve the last problem. To do so, we should compute the gradients of the constraints $C_k$,$f$,$g$ and the objective function $\phi$. 

Since $C_k,f$ and $g$ are explicitly expressed via usual functions of $(x_1,\dots,y_N)$, we obtain explicit formulas for the gradients by straightforward computations. This is not the case for the objective function $\phi$, for which we use a shape derivation result proved in \cite{parini_saintier}. Since $\Omega$ is convex, it admits a unique Cheeger set $C_\Omega$ that is $C^{1,1}$ (see \cite{cham_1}). We are then in position to use the following result of \cite[Corollary 1.2]{parini_saintier}:
$$h'(\Om,V) := \lim\limits_{t\rightarrow 0}\frac{h(\Om_t)-h(\Om)}{t}= \frac{1}{|C_\Om|} \int_{\partial C_\Om\cap \partial \Om} \big(\kappa -h(\Om)\big)\langle V,n \rangle d\mathcal{H}^1,$$
where $V\in \R^2\longrightarrow \R^2$ is a smooth perturbation, $\Om_t := (Id+tV)(\Om)$ (where $Id:x\longmapsto x$ is the identity map), $n(x)$ is the normal to $\partial \Om$ at the point $x$ and $\kappa(x)$ is the curvature of $\partial \Om$ at the point $x$. 

Since $\Om$ is a convex polygon and $C_\Om$ is $C^{1,1}$, the contact set $\Om\cap C_\Om$ is given by a finite union of segments. We then have $\kappa = 0$ on $\partial C_\Om \cap \partial \Om$. Thus, if we denote { by} $V_{x_k}$ and $V_{y_k}$ the perturbations respectively associated to the variables $x_k$ and $y_k$, where $k \in \llbracket 1,N\rrbracket$, we have
 $$
\left \{
\begin{array}{c @{=} c}
    \frac{\partial \phi}{\partial x_k}(x_1,\dots,x_N,y_1,\dots,y_N)\ \ \ &\ \ \ -\frac{h(\Om)}{|C_\Om|} \int_{\partial C_\Om \cap \partial \Om} \langle V_{x_k},n \rangle d\mathcal{H}^1, \vspace{3mm}\\ 
    \frac{\partial \phi}{\partial y_k}(x_1,\dots,x_N,y_1,\dots,y_N)\ \ \ &\ \ \ -\frac{h(\Om)}{|C_\Om|} \int_{\partial C_\Om \cap \partial \Om} \langle V_{y_k},n \rangle d\mathcal{H}^1.
\end{array}
\right.
    $$   
    
\subsection{Results}
In Figure \ref{fig:blaschke5}, we plot the points corresponding to $10^5$ random convex pentagons and the points corresponding to the optimal pentagons obtained for $p_0\in\{P(R_5)+0.02\cdot k\ |\ k\in \llbracket 0,20\rrbracket \}$, in addition to the  graphs of the functions $x\longmapsto x$ and $x\longmapsto f_5(x):= \frac{x+\sqrt{x^2+4(\pi-5\tan{\frac{\pi}{5}})}}{2}$ whose hypographs represent the inequalities \eqref{eq:trivial} and \eqref{eq:uper_bound} in the Blaschke--Santal\'o diagram.  The random polygons are generated by using the algorithm presented in \cite{sander} (based on a work of P. Valtr, see \cite[Section 4]{random_polygon}). The Cheeger constants are computed by a Matlab algorithm implemented by B. Bogosel, see \cite{beni}.

 
\begin{figure}[h]
    \centering
    \includegraphics[scale=.38]{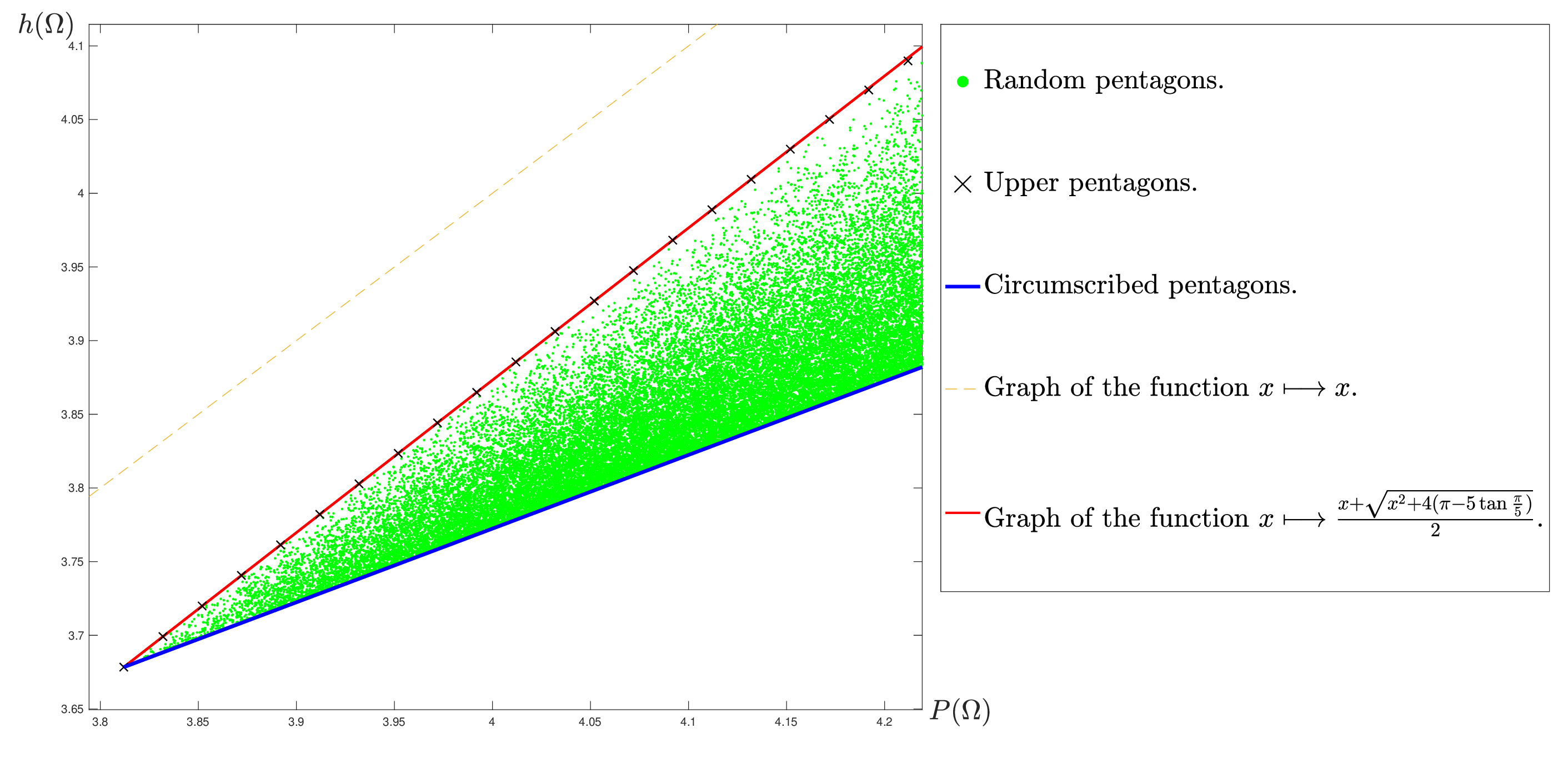}
    \caption{Numerical approximation of the Blaschke--Santal\'o diagram of convex pentagons.}
    \label{fig:blaschke5}
\end{figure}

In Figure \ref{fig:zoom}, we provide a zoom on the upper boundary of the diagram $\D_5$. We observe that the points $(p,g_5(p))_{p\ge P(R_5)}$ (where $g_5$ is introduced in Definition \ref{def:boundary_polygons}) are  at first exactly located on the red and continuous curve corresponding to the graph of the function $x\longmapsto f_5(x):= \frac{x+\sqrt{x^2+4(\pi-5\tan{\frac{\pi}{5}})}}{2}$, then { they detach} and become strictly lower than it. We also note that the abscissa $c_5$ introduced in the statement of Theorem \ref{th:diagram_polygon}  is indeed (as shown in \eqref{ineq:c_N}) bounded from above by $10\sqrt{\tan \frac{3\pi}{10}}\approx 11.73$.

\begin{figure}[h!]
    \centering
    \includegraphics[scale=.6]{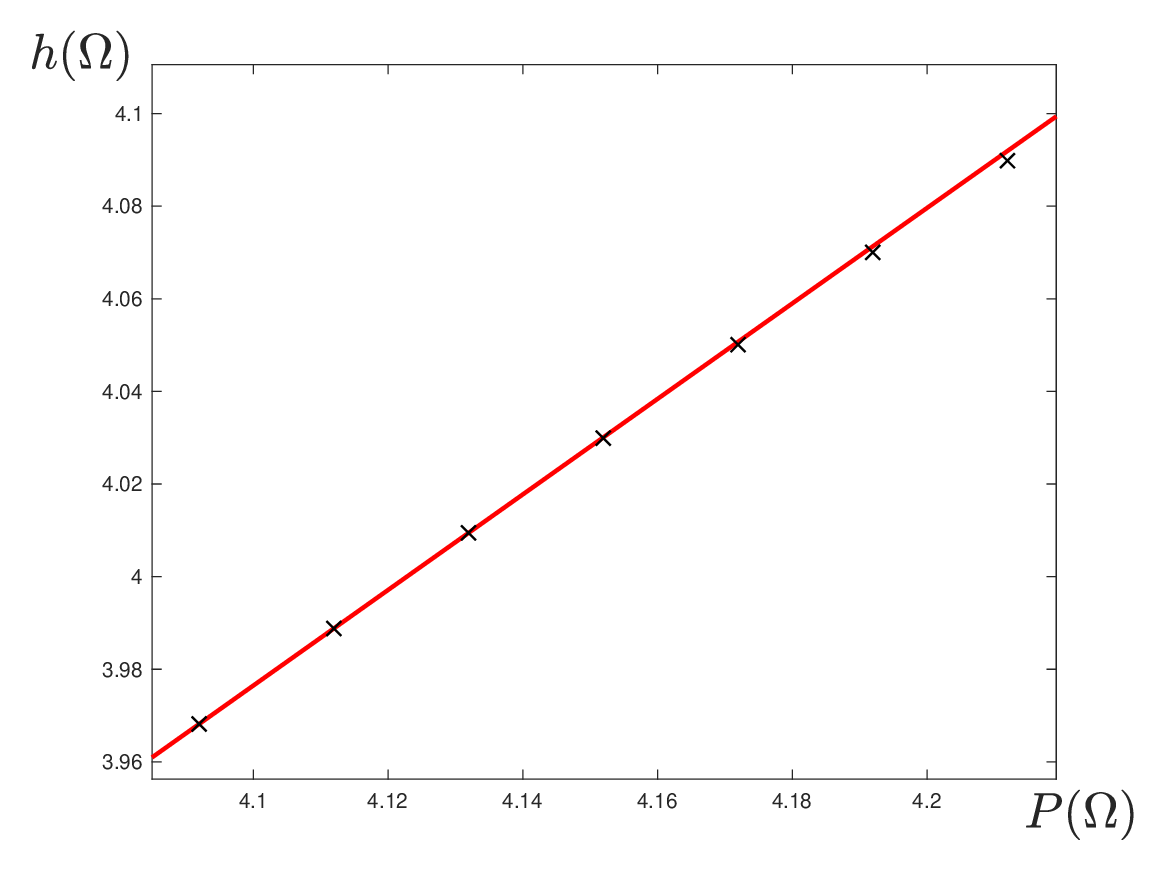}
    \caption{A zoom on the upper boundary of the diagram $\D_5$.}
    \label{fig:zoom}
\end{figure}

Finally, in Figure \ref{fig:polygons}, we give the obtained optimal pentagons (solutions of \eqref{eq:problem}) for $p_0\in\{3.86,4,5\}$. We note that for larger values of $p_0$, the maximizers seem not to be Cheeger-regular.

\begin{figure}[!h]
    \centering
\begin{tabular}{|*{4}{>{\centering}m{3.5cm}|}}
\hline
Values of $p_0$ & 3.86& 4 & 5\tabularnewline
\hline
Optimal pentagons& \includegraphics[scale=0.21]{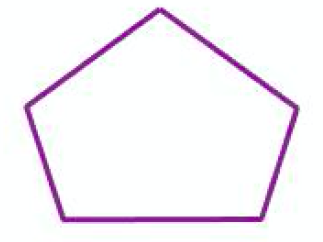} & \includegraphics[scale=0.21]{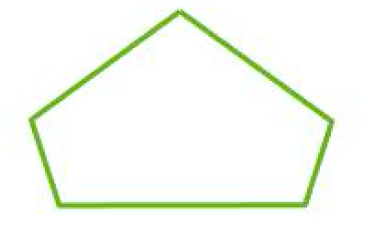}&  \includegraphics[scale=0.21]{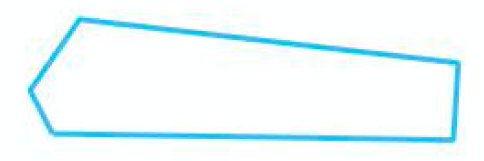}\tabularnewline
\hline
\end{tabular}
    \caption{Numerically obtained (upper) extremal pentagons corresponding to different values of $p_0$.}
    \label{fig:polygons}
\end{figure}

{
\begin{remark}
Our numerical approach has been tested on problems for which the solutions are theoretically provided in the present paper (see Theorem \ref{th:diagram_polygon} and its proof in Section \ref{th:diagram_polygon}). Namely, problems of the type 
$$\min\{h(\Om)\ |\ P(\Om)=p_0,\ |\Om|=1\  \text{and}\ \Om\in \p_N\} = \frac{p_0}{2}+\sqrt{\pi},$$ 
where $N\ge 3$ and $p_0\ge P(R_N)$ and
$$\max\{h(\Om)\ |\ P(\Om)=p_0,\ |\Om|=1\  \text{and}\ \Om\in \p_N\}=f_N(p_0)=\frac{p_0+\sqrt{p_0^2+4(\pi-N\tan{\frac{\pi}{N}})}}{2},$$
where $N$ is even. 
\end{remark}
}

\newpage 
\section{Some applications} \label{s:remarks}
In this last section, we give two applications of the results of the present paper. 
\subsection{Improvement of a classical lower bound for the Cheeger constant of polygons}\label{ss:improvement_of_inequalities}
One early result in the spirit of the inequality \eqref{eq:new_inequality} is due to R. Brooks and P. Waksman \cite{MR883424}, see Proposition \ref{th:brooks} below. It gives a lower estimate of the Cheeger constant of convex polygons, which we show to be a consequence of the inequality \eqref{eq:new_inequality}. 
\begin{Proposition}\cite[Theorem 3.]{MR883424}\label{th:brooks}
If $\Om$ is a convex polygon, we denote { by} $\Om^*$ the (unique up to rigid motions) circumscribed polygon which has the same area as $\Om$ and whose angles are the same as those of $\Om$, then 
\begin{equation}\label{eq:brook}
h(\Om)\ge h(\Om^*)= \frac{\sqrt{T(\Om)}+\sqrt{\pi}}{\sqrt{|\Om|}},
\end{equation}
where the functional $T$ is defined in \eqref{eq:def_T}.
\end{Proposition}
\begin{proof}
We use the inequality \eqref{eq:new_inequality} to provide an alternative proof: 
$$h(\Om) \ge \frac{P(\Om)+\sqrt{4\pi |\Om|}}{2|\Om|}\ge \frac{2\sqrt{|\Om|}\sqrt{T(\Om)}+\sqrt{4\pi|\Om|}}{2|\Om|}=\frac{\sqrt{T(\Om)}+\sqrt{\pi}}{\sqrt{|\Om|}}=\frac{\sqrt{T(\Om^*)}+\sqrt{\pi}}{\sqrt{|\Om^*|}},$$
where we respectively used \eqref{eq:new_inequality} and \eqref{eq:fragala} for the first and second inequalities and the fact that $\Om^*$ has the same area and interior angles as $\Om$ for the last equality. 



\end{proof}

\subsection{On the stability of the Cheeger constant of polygons}
We use the inequality \eqref{eq:new_inequality} and \cite[Proposition 2.1]{MR3550852} to give a quantitative version of the polygonal Faber-Krahn type inequality for convex polygons (see \cite{bucur_fragala}). 
\begin{Proposition}
Take $N\ge 3$. There exists a positive constant $C_N$ such that for every convex $N$-gon $\Om$ with unit area, there exists a rigid motion $\rho$ of $\R^2$ such that
\begin{equation}\label{eq:quant}
h(\Om)^2-h(R_N)^2\ge C_N d^H\big(\Om,\rho(R_N)\big)^2,
\end{equation}
where $R_N$ is a regular polygon of unit area and $N$ sides. 
\end{Proposition}
\begin{proof}
Take $N\ge 3$ and $\Om$ a convex $N$-gon. We have by the inequality \eqref{eq:new_inequality} applied to $\Om$ and the equality \eqref{eq:egal} applied to the circumscribed polygon $R_N$
$$P(\Om)\leq 2(h(\Om)-\sqrt{\pi})\ \ \ \ \ \ \text{and}\ \ \ \ \ \ P(R_N)= 2(h(R_N)-\sqrt{\pi}).$$ 
Thus,
\begin{align*}
P(\Om)^2- P(R_N)^2&\ \ \ \ \ \leq\ \ \ \  4(h(\Om)-\sqrt{\pi})^2-4(h(R_N)-\sqrt{\pi})^2  \\
&\ \ \ \ \ =\ \ \ \   4\big(h(\Om)^2-h(R_N)^2-2\sqrt{\pi}\big(h(\Om)-h(R_N)\big)\big)\\
&\ \ \ \ \ \leq\ \ \ \  4\big(h(\Om)^2-h(R_N)^2\big),
\end{align*}
where the last inequality is a consequence of the polygonal Faber-Krahn type inequality $h(\Om)\ge h(R_N)$, see \cite{bucur_fragala}. 

On the other hand, it is proved in \cite[Proposition 2.1]{MR3550852} that there exists $C_N>0$ depending only on $N$ such that 
$$P(\Om)^2- P(R_N)^2\ge 4C_Nd^H\big(\Om,\rho(R_N)\big)^2.$$
Finally, by combining the latter inequalities, we get the announced result.
\end{proof}
{
\begin{remark}
The quantitative inequality \eqref{eq:quant} shows in particular the stability of the Cheeger constant in the neighborhood of regular polygons among convex polygons with the same number of sides and the same area. In the sense that if the Cheeger constants of a convex polygon and a regular polygon with  the same number of sides and the same area are close, then the polygon looks (up to rigid motions) like the regular one. A similar result can be obtained for non convex $N$-gons, see  \cite{MR3436708}.
\end{remark}

\section*{Acknowledgments}

\hspace{.3cm} The author would like to warmly thank the anonymous referee for his careful reading and detailed report that considerably helped to improve the presentation of the material of the paper. The author would also like to thank Dorin Bucur and Jimmy Lamboley for stimulating conversations and valuable comments.

\vspace{1mm}

This work was partially supported by the project ANR-18-CE40-0013 SHAPO financed by the French Agence Nationale de la Recherche (ANR) and by the Alexander von Humboldt Foundation. 

}

%










\bibliographystyle{plain}
\bibliography{samplee}

\vspace{1cm}
(Ilias Ftouhi) \textsc{ Friedrich-Alexander-Universit{\"a}t Erlangen-N{\"u}rnberg, Department of Mathematics, Chair of Applied Analysis (Alexander von Humboldt Professorship), Cauerstr. 11, 91058 Erlangen, Germany.}\vspace{1mm}

\textit{Email address:} \textbf{\texttt{ilias.ftouhi@fau.de}}







\end{document}